\documentclass[11pt, a4paper, twoside]{amsart}
\usepackage{mathrsfs, tabularx}
\usepackage{amsthm, amsmath, amssymb, amscd, float, latexsym, times, epic, eepic}
\usepackage{graphicx,epsfig}

\def\expec{\mathbb{E}}

\newtheorem {lemma}{Lemma}
\newtheorem {proposition}{Proposition}
\newtheorem {theorem}{Theorem}
\newtheorem {definition}{Definition}

\newtheorem {corollary}{Corollary}

\newtheorem {remark}{Remark}

\def \Z {{\Bbb{Z}}}
\def \C {{\Bbb{C}}}
\def \N {{\Bbb{N}}}
\def \R {{\Bbb {R}}}

\def \dist {\hbox{dist}}
\def\bra{\langle}
\def\cet{\rangle}
\def \bv {{\bf {v}}}
\def \bV {{\bf {V}}}
\def \bu {{\bf {u}}}
\def \bU {{\bf {U}}}
\def \bm {{\bf {m}}}
\def \bw {{\bf {w}}}

\def \onef{{\bf 1}}

\newcommand{\norm}[1]{\left\|#1 \right\|}
\newcommand{\trace}{\ensuremath{{{\rm Tr}}}}

\renewcommand{\l}{\left}
\renewcommand{\r}{\right}

\newcommand{\abs}[1]{\left| #1 \right|}

\newcommand{\vek}[2]{{ #1 \choose #2}}

\newcommand{\ep}{\epsilon}

\newcommand{\T}{\mathbb{T}}
\newcommand{\ran}{{\rm Ran}}

\newcommand{\Pro}{\mathbb{P}}
\newcommand{\Ec}{{\mathcal E}}
\newcommand{\Ll}{{\mathcal L}}
\newcommand{\apr}[1]{\langle #1 \rangle}

\newcommand{\td}{D_q}

\title[On infinite-dimensional hierarchical probability models]{On infinite-dimensional hierarchical probability models in statistical inverse problems}

\author[Tapio Helin]{}

\begin{document}

\maketitle

 \centerline{\scshape Tapio Helin}
 \medskip
 {\footnotesize
   \centerline{Department of Mathematics and System Analysis}
   \centerline{Helsinki University of Technology}
   \centerline{P.O. Box 1100 (Otakaari 1 M),
 FI-02015 TKK, Finland} }
 \medskip

\begin{abstract}
In this article, the solution of a statistical inverse problem $M = AU+\Ec$ by the Bayesian approach is studied where
$U$ is a function on the unit circle $\T$, i.e., a periodic signal. The mapping $A$ is a smoothing linear operator
and $\Ec$ a Gaussian noise. The connection to the solution of  
a finite-dimensional computational model $M_{kn} = A_k U_n + \Ec_k$ is discussed.
Furthermore, a novel hierarchical prior model for obtaining edge-preserving conditional mean estimates is introduced.
The convergence of the method with respect to finer discretization is studied and
the posterior distribution is shown to converge weakly.
Finally, theoretical findings are illustrated by a numerical example with simulated data.
\end{abstract}

%------------------------------------
% INTRODUCTION
%------------------------------------

\section{Introduction}

Reconstruction methods with edge-preserving or -enhancing properties
are widely studied topic in deterministic inverse problems.
There exists a variety of different sophisticated approaches in the literature including
functional regularization (e.g., the total variation approach \cite{ROF}) 
or geometrical methods (e.g., the level set methods \cite{Se}).
In the Bayesian inversion theory some methods 
have been introduced aiming for an edge-preserving point estimate
in the finite-dimensional setting \cite{GG,BoSa,CGL,Ash}.
Especially the work by 
Calvetti and Somersalo \cite{CS,CS2} with hierarchical priors is closely related to this paper.
In general it seems to be difficult to establish how the posterior distribution behaves 
asymptotically, i.e., as discretization of the problem gets finer.
This is due to the fact that such methods usually require non-Gaussian prior modeling and the related infinite-dimensional 
Bayesian theory is not fully developed. 
This paper introduces a novel hierarchical structure
leading to non-Gaussian prior modeling for signal segmenting problems. 
We show that the limiting behavior of our model can be analyzed. 

Let us discuss the current perspectives in Bayesian modeling.
Consider a linear inverse problem
\begin{equation}
  \label{idealmodel}
  M = A U + \Ec
\end{equation}
where $U$ is the object of interest, $\Ec$ a noise and $M$ the measured data on some function spaces. 
In the Bayesian inversion these quantities are modeled as random variables
and their probability distributions depict all information available prior to the
measurement. With this information the goal is to make statistical inference on $U$ given the
model equation \eqref{idealmodel} and a realization $M(\omega_0)$ of $M$. 
Sometimes the prior distribution of the object of interest $U$ depends on an unknown parameter which then
becomes part of the modeling and inference problem. Such prior structures are often referred to as {\em hierarchical models}.

In practice the measurement is 
often produced by some finite-dimensional projection $M_k=P_k M$.
Furthermore, one also has to discretize $U$ for computational purposes. 
This yields the {\em computational model}
\begin{equation}
  \label{compmodel}
  M_{kn} = P_k(A U_n + \Ec) = A_k U_n + \Ec_k.
\end{equation}
Notice the two independent discretization levels $n$ and $k$.
Solving the inverse problem with the Bayesian approach requires two steps: first, 
one translates all a priori information into the probability
distributions of $U_n$ and the noise $\Ec_{k}$.
The {\em posterior probability} ${\mathcal P}_{kn}(\cdot\;|\; m)$, i.e., the probability distribution of $U_n$ 
conditioned on the measurement $m=M_k(\omega_0)$, is then obtained by using the Bayes formula and equation \eqref{compmodel}.

Usually the ultimate goal is to compute some information, e.g., point or spread estimates, 
from the posterior distribution. A point estimate that we discuss frequently in this paper is
the {\em conditional mean} (CM) estimate which can be written for equation \eqref{compmodel} in 
Euclidian spaces $\R^n$ and $\R^k$ as
\begin{equation}
  \label{cm_finite}
  u^{CM}_{kn} = \int_{\R^n} u\; d{\mathcal  P}_{kn}(u\;|\; m).
\end{equation}
Now a natural question follows: what happens to the reconstructed information 
if $U_n$ or $\Ec_k$ is modeled on finer
discretization, i.e., with a bigger $n$ or $k$? 
Moreover, do the posterior probability distributions converge and 
how to guarantee that 
the reconstructed objects stay stable (e.g., CM estimate converges) as $n$ and $k$ increase?

The interplay between solutions of problems \eqref{idealmodel} and \eqref{compmodel}
in general situations is not fully understood. However, some partial results exist.
In fact, if $U_n$ and $\Ec_k$ are obtained by projections from Gaussian distributions 
the convergence of posterior distribution has been proved in very general setting by Lasanen in \cite{La}.
To the author's knowledge only convergence studies with non-Gaussian posterior distribution
have been done from this point of view recently in \cite{Pi} and \cite{LSS}.
These first positive results show some general conditions for obtaining weakly converging 
posterior distributions and in addition converging CM estimates. We emphasize that
these results require Gaussian noise distribution.

Yet another non-trivial question is how to make sure that the crucial statistical 
properties of posterior distribution are not lost asymptotically?
This is highly relevant to the edge-preserveness discussed above.
Namely, in \cite{LS} it was shown that the usual modeling of TV prior carries an unpleasant defect such that the
edge-preserving property is lost from the CM reconstructions as dimensionality of the problem increases.
The reason behind this is that under different parameterization the prior distribution either converges to
a Gaussian smoothness prior or diverges. 
In \cite{LSS} a non-Gaussian prior structure is proposed
for edge-preserving CM estimates. The estimates $u^{CM}_{kn}$ are shown to converge to
so-called {\em reconstructors} that generalize the concept of CM estimates in infinite-dimensional spaces.
We discuss this in more details later. 
The work by Piiroinen in \cite{Pi} contains results about 
the existence of a discretization leading to converging posterior information in general non-Gaussian setting.

Let us now review other related literature on the topic. First results on the Bayesian inversion
in infinite-dimensional function spaces were introduced in \cite{Fr} by Franklin. This 
research has then been continued and generalized by Mandelbaum \cite{Ma},
Lehtinen, P\"aiv\"arinta and Somersalo \cite{LPS}, Fitzpatrick \cite{Fi}, and Luschgy \cite{Fi}.
Lastly, we want to stress that the convergence of posterior distributions has also been studied from different perspectives.
Namely, in \cite{HP,HP2,NP} such convergence is studied when the
objective information becomes more accurate. Also, model reduction problems
are considered in \cite{kaipiosomersalo}.
For a general presentation on the Bayesian inverse problems theory and computation see \cite{kaipiosomersalo2} and \cite{CS3}. The topic of probability theory in Banach spaces is covered in \cite{VTI}.

This paper studies the problem of edge-preserving reconstructions in
signal restoration problems with the emphasis on how to locate discontinuities. 
For technical reasons we concentrate on periodic signals, i.e., the domain for our study is a 1-dimensional sphere $\T$.
We model our prior beliefs of the unknown signal $u$ with a hierarchical structure $(U,V)$ where the auxiliary random 
variable $V$ models how the discontinuities are distributed. The conditional distribution of $U$ given a sample of $V$
then models our prior information about $u$ if we know where the discontinuities are located.
Such Bayesian modeling has close connection to previous hierarchical segmentation methods \cite{GG,CS,CS2}.
The method draws also a lot of inspiration from the celebrated Mumford--Shah image segmentation method
\cite{mumfordshah} and its variational approximation introduced by Ambrosio and Tortorelli \cite{AT1,AT2}.

In this paper we introduce a finite-dimensional
prior structure $(U_n,V_n)$ that produces a weakly converging posterior structure
in the presence of a Gaussian noise. The main 
theoretical results concerning the prior can be divided into three parts:
\begin{itemize}
\item[(i)] There exists a well-defined random variable $(U,V):\Omega\to L^2(\T)\times L^2(\T)$
to which $(U_n,V_n)$ converges in distribution.
\item[(ii)] The posterior distributions ${\mathcal P}_{kn}$ converge weakly in $L^2(\T)\times L^2(\T)$
assuming that the measurements converge.
\item[(iii)] The CM estimate $(u_{kn}^{CM},v_{kn}^{CM})$ converges to reconstructors of problem \eqref{idealmodel}.
\end{itemize}
In addition we improve the results in \cite{LSS} concerning the general theory.
We implement our method in practice and include some numerical examples
with computer generated data.
The connection of maximum a posteriori (MAP) estimates to Ambrosio-Tortorelli minimizers 
that was presented in \cite{HeLa} is not studied here.

This paper is organized as follows. In Section 2 we introduce relevant concepts and main results concerning the general theory.
The infinite-dimensional hierarchical prior model $(U,V)$ in $L^2(\T)\times L^2(\T)$ is defined in Section 3. We carefully show that such a construction is well-defined.
Discretized prior distributions for $(U_n,V_n)$ are constructed in Section 4. It is important
to note that we can explicitly write down the related density functions. This becomes highly valuable
in numerical implementation as no more approximations need to be made.
Section 5 is divided into three parts. First the theorems of Section 2 are proved.
Secondly, we show here that $(U_n,V_n)$ converges to $(U,V)$ in distribution on $L^2(\T)\times L^2(\T)$.
We conclude Section 5 by showing the important property of uniformly finite exponential moments for
the introduced prior structure.
Finally in Section 6 we illustrate with numerical examples how our method works in practice.

%------------------------------------
% PROBLEM SETTING
%------------------------------------

\section{General setting}

Next we define problem \eqref{idealmodel} rigorously. In order to do so let us introduce some notations.
Below $\bra\cdot,\cdot\cet$ refers to pairing of generalized functions with test functions.
In real Banach space $B$ the dual pairing is denoted by $\bra\cdot, \cdot\cet_{B'\times B}$.
In a real Hilbert space $H$ we denote the inner product by $\bra\cdot, \cdot\cet_H$.
We denote the Borel sets in $B$ by ${\mathcal B}(B)$.
Throughout this paper whenever not explicitly mentioned we assume 
the measurable structure of Borel sets.
The notation $\Ll(B_1,B_2)$ stands for the space of bounded linear operators between Banach spaces
$B_1$ and $B_2$, and $\Ll(B,B)$ is abbreviated as $\Ll(B)$. If the operator $T:B_1\to B_2$ is a bounded linear
operator, we denote the adjoint operator by $T':B_2'\to B_1'$.
Recall also that a bounded linear operator $T$ in a Hilbert space $H$ is said to be in the trace class if
\begin{equation*}
  \trace_H (T) := \sum_{j=1}^\infty \bra T e_j, e_j\cet_H < \infty
\end{equation*}
for some orthonormal basis $\{e_j\}_{j=1}^\infty \subset H$. 
We want to point out that the definition is independent of the
choice of the basis.
Throughout the paper if not explicitly mentioned $C$ denotes
a positive constant.
For two functions $f,g : X \to \R \cup \{\infty\}$ we also write
$f\preceq g$ if
there exists a constant $C>0$ such that $f\leq Cg$ as functions. 
Finally, for any $s\in\R$, let $H^s(\T)$ be the
$L^2$-based Sobolev space \cite{Adams} equipped with Hilbert space inner product
\begin{equation*}
  \bra \phi,\psi \cet_{H^s} = \int_{\T} ((I-\Delta)^{s/2}\phi)(x) ((I-\Delta)^{s/2}\psi)(x) dx
\end{equation*}
for any $\phi,\psi \in H^s(\T)$.

Let us return to considering problem \eqref{idealmodel}. Let $(\Omega,\Sigma,\Pro)$ be a complete probability space
with a product structure $\Omega = \Omega_{pr} \times \Omega_{er}$, $\Sigma = \overline{\Sigma_{pr}\otimes \Sigma_{er}}$
and $\Pro = \Pro_{pr}\otimes \Pro_{er}$. Throughout this section $H$ will be fixed to denote a real separable Hilbert space.
We assume the following conditions:
\begin{itemize}
\item[(i)] The mapping $U :\Omega_{pr}\to H$ is a random variable.
\item[(ii)] The mapping $A: H \to H^1(\T)$ is a bounded linear operator.
\item[(iii)] The random variable $\Ec:\Omega_{er}\to H^{-1}(\T)$ is Gaussian with expectation $\expec \Ec = 0$
and a covariance operator $C_\Ec : H^{-1}(\T) \to H^{-1}(\T)$.
\item[(iv)] The range of $C_\Ec$ is dense in $H^{-1}(\T)$.
\end{itemize}
The conditions (iii) and (iv) imply that $C_\Ec$ is one-to-one, self-adjoint
and in the trace class and that 
we have a unique positive and self-adjoint power $C_\Ec^t$ for any $t\in\R$.
Later in numerical examples $\Ec$ has a
covariance operator $C_\Ec = (I-\Delta)^{-1}:H^{-1}(\T)\to H^1(\T)$. Such a random variable
is white noise in the sense of generalized random variables \cite{La}.
\begin{definition}
  Let $\mu$ be a centered Gaussian measure on $(H,\mathcal{B}(H))$ and its covariance operator
  $C : H\to H$ such that ${\rm Ran}(C)$ is dense in $H$. We call the
  real separable Hilbert space
  \begin{equation*}
    H(\mu) = \left\{ f \in H \; | \; \norm{C_X^{-1/2} f}_{H} < \infty\right\}
  \end{equation*}
  equipped with inner product
  \begin{equation*}
    \bra f,g\cet_{H(\mu)} = \bra C_X^{-1/2} f, C_X^{-1/2} g \cet_H
  \end{equation*}
  for any $f,g\in H(\mu)$ 
  the {\em Cameron-Martin space} (or the {\em reproducing kernel Hilbert space}) of $\mu$.
\end{definition}
This definition can be seen to coincide with the usual definition of Cameron-Martin spaces
by Proposition 2.9 in \cite{DaPrato}.
The Cameron-Martin space structure is used later in Section 4.
For an extensive presentation on the topic in locally
convex spaces see \cite{B}.

If $U\in L^1(\Omega,\Sigma; H)$ and $\Sigma_0$ is
a sub $\sigma$-algebra of $\Sigma$, we denote the conditional expectation 
of $U$ with respect to $\sigma$-algebra $\Sigma_0$ by
$\expec (U | \Sigma_0)$. That is, $\expec (U | \Sigma_0) \in L^1(\Omega, \Sigma_0 ; H)$ and it satisfies
\begin{equation}
  \int_D \expec(U | \Sigma_0)(\omega) \Pro (d\omega)
  = \int_D U(\omega) \Pro(d\omega)
  \quad \textrm{for all } D\in \Sigma_0.
\end{equation}

All vector-valued integrals in this work are standard Bochner integrals. 
For more information on Bochner integrals see \cite{DU}.
The operator $P_{\Sigma_0}: U \mapsto \expec (U | \Sigma_0)$ is
a projection $P_{\Sigma_0} : L^1(\Omega, \Sigma; H) \to L^1(\Omega, \Sigma_0 ; H)$, where
$L^1(\Omega,\Sigma_0; H)$ denotes the space of measurable functions from $(\Omega,\Sigma_0)$
to $(H,{\mathcal B}(H))$ which are Bochner integrable.

\begin{definition}
Denote by ${\mathcal M} \subset \Sigma$ the 
$\sigma$-algebra generated by the random variable $M$.
We say that any deterministic function
\begin{equation}
  {\mathcal R}_M(U | \cdot ) : H^{-1}(\T)\to H, \quad m 
  \mapsto {\mathcal R}_M (U | m),
\end{equation}
is a {\rm reconstructor} of $U\in L^1(\Omega, \Sigma; H)$ with measurement $M$ if
\begin{equation}
  {\mathcal R}_M(U | M(\omega)) = \expec(U | {\mathcal M})(\omega) \quad \textrm{almost surely}.
\end{equation}
If $\widetilde{H}$ is a real separable Hilbert space, 
$g: (H,{\mathcal B}(H)) \to (\widetilde{H},\mathcal{B}(\widetilde{H}))$ 
is a measurable function and $g(U)\in L^1(\Omega,\Sigma ; \widetilde{H})$, we define
${\mathcal R}_M (g(U) | \cdot) : H^{-1}(\T) \to \widetilde{H}$ to be any deterministic function satisfying
\begin{equation}
  {\mathcal R}_M(g(U) | M(\omega)) = \expec(g(U) | {\mathcal M})(\omega) \quad \textrm{almost surely}.
\end{equation}
\end{definition}

We refer to \cite{LSS} for the existence of ${\mathcal R}_M$. Note that although
${\mathcal R}_M$ is not necessarily unique it was shown in \cite{LSS} that
in the presence of Gaussian noise the following choice can be made:
Assume that the prior distribution $\lambda$ of $U$ has finite exponential moments, i.e.,
\begin{equation*}
  \int_H \exp(c\norm{u}_H) d\lambda(u) < \infty
\end{equation*}
for any $c\in\R$,  and
assume $\widetilde H$ is a real separable Hilbert space. 
Furthermore, let $g: (H,\mathcal{B}(H)) \to (\widetilde H,\mathcal{B}(\widetilde H))$ be a measurable function 
satisfying $\expec \norm{g(U)}_{\widetilde H}<\infty$. Then
a function ${\mathcal R}_M(U | \cdot ) : H^{-1}(\T)\to H$ defined by formula
\begin{equation}
  \label{recon_formula}
  {\mathcal R}_M (g(U)\; | \;m) =
  \frac{\int_H g(u) \Xi(u,m) d\lambda(u)}{\int_H \Xi(u,m) d\lambda(u)}
\end{equation}
is a reconstructor, where $\Xi: H\times H^{-1}(\T) \to\R$ is the function
\begin{equation*}
  \Xi(u,m) = \exp(-\frac 12 \norm{Au}^2_{L^2}+ \bra C_\Ec^{-1} Au, m\cet_{H^{-1}}).
\end{equation*}
Throughout this paper we make the above choice of reconstructors.

As was discussed earlier
the measurement is never infinite-dimensional in practice. Let us next explain how we assume
the measurement to be obtained.
\begin{definition}
  \label{propermeasurement}
  The finite-dimensional linear projections $P_k : H^{-1}(\T)\to H^{-1}(\T)$, $k\in\N$, are called
  {\em proper measurement projections} when they satisfy the
  following conditions:
  \begin{itemize}
    \item[(i)] We have $\ran (P_k)\subset H^1(\T)$ and
      $\norm{P_k}_{\Ll(H^1)} \leq C_0$ for some
      constant $C_0$ with all $k\in\N$.
    \item[(ii)] For $t\in \{-1,1\}$ we have
    \begin{equation*}
      \lim_{k\to\infty} \norm{P_k f-f}_{H^t} = 0
    \end{equation*}
    for all $f\in H^t(\T)$.
    \item[(iii)] For all $\phi,\psi\in L^2(\T)$ it holds that
    \begin{equation*}
      \bra P_k \phi, \psi\cet_{L^2} = \bra \phi,P_k \psi\cet_{L^2}.
    \end{equation*}
  \end{itemize}
\end{definition}
The conditions in Definition \ref{propermeasurement} are same as
in \cite[Thm. 3]{LSS} and are motivated there. We note that in this paper
these assumptions are only used in the proof of Theorem \ref{LSS_result}.

In practical situation the measurement is a realization of a random variable
\begin{equation}
  \label{apprmodel}
  M_k = P_k M = A_k U + \Ec_k,
\end{equation}
where $A_k = P_kA$, $\Ec_k = P_k \Ec$. 
In order to be able to compute a numerical solution one has to discretize
also the random variable $U$ (independently of $P_k$) in $H$. 
Denote the discretization by $U_n : \Omega \to H_n\subset H$ in a finite-dimensional subspace $H_n$. Now the two
discretizations with respect to $n$ and $k$ lead to the computational
model \eqref{compmodel}. We note that the reconstructor can be defined for all above models, 
for problem \eqref{idealmodel} on $H^{-1}(\T)$ and for problems \eqref{compmodel} and \eqref{apprmodel} on ${\rm Ran} (P_k)$. 
Before next definition recall that probability 
measures $\mu_n$, $n\in\N$, {\em converge weakly} to $\mu$ in $(H,{\mathcal B}(H))$
if for every bounded and continuous function $f: H \to \R$ it holds that
\begin{equation*}
  \lim_{n\to\infty}\int_H f(u) d\mu_n(u) = \int_H f(u) d\mu(u).
\end{equation*}
In the following definition we characterize a condition that allows converging probability measures to have
only very small tails.
\begin{definition}
  \label{unifprior}
  We call measures $\mu$ and $\mu_n$, $n\in\N$, on $(H,\mathcal{B}(H))$ 
  {\em uniformly discretized with exponential weights} if
  \begin{itemize}
    \item[(i)] $\mu_n$ converges weakly to $\mu$ on $H$ and
    \item[(ii)] for every $b>0$ there exists a constant $0<C(b)<\infty$ such that
    \begin{equation*}
      \int_H \exp(b \norm{u}_H) d\mu_n(u) \leq C(b) \quad {\rm and}\quad
      \int_H \exp(b \norm{u}_H) d\mu(u) \leq C(b)
    \end{equation*}
    for every $n\in\N$.
  \end{itemize}
\end{definition}

We are now ready to formulate our main theorem regarding the general theory. We postpone the proof to Section 5.1.
\begin{theorem}
  \label{LSS_result}
  Assume the following three conditions: 
  \begin{itemize}
    \item[(i)] The operators $P_k: H^{-1}(\T)\to H^{-1}(\T)$, $k\in\N$, are proper measurement projections. 
    \item[(ii)] The probability distributions of $U_n,U : \Omega\to H$, $n\in\N$, are uniformly discretized with
    exponential weights.
     \item[(iii)] A continuous function $g: H\to \widetilde{H}$ 
       where $\widetilde{H}$ is a real separable Hilbert space, satisfies
       \begin{equation*}
         \norm{g(u)}_{\widetilde{H}} \leq C \exp(C\norm{u}_H)
       \end{equation*}
       for all $u\in H$ with some constant $C$.
  \end{itemize}
  Now let $u=U(\omega_0)$ and $\epsilon=\Ec(\omega_0)$ be realizations of the random variables $U$ and $\Ec$,
  respectively, and let
  \begin{equation*}
    m = Au+\epsilon \quad {\rm and} \quad m_k = A_k u + P_k \epsilon
  \end{equation*}
  be the realizations of the random variables $M$ and $M_k$ in equations \eqref{idealmodel} and \eqref{apprmodel},
  respectively. Then the reconstructors defined by
  formula \eqref{recon_formula} for models \eqref{idealmodel} and \eqref{apprmodel} satisfy
  \begin{equation*}
    \lim_{k,n\to\infty} {\mathcal R}_{M_{kn}} (g(U_n) \; | \; m_k) = {\mathcal R}_{M} (g(U) \; | \; m)
  \end{equation*}
  in $\widetilde{H}$.
\end{theorem}
Let $E\subset H$ be a Borel set and $\onef_E$ be the indicator function of $E$. Define probability measures
\begin{eqnarray*}
  {\mathcal P}(E\; |\; m) & = & {\mathcal R}_M(\onef_E(U)\; |\; m), \\
  {\mathcal P}_{kn} (E\; |\; m_k) & = & {\mathcal R}_{M_{kn}} (\onef_E(U)\; |\; m_k)
\end{eqnarray*}
on $H$ with the same choices of reconstructor made in Theorem \ref{LSS_result}.
One notices that these measures correspond to the posterior distribution obtained from Bayes formula
in the finite-dimensional case.
An important corollary to Theorem \ref{LSS_result} is shown in \cite{LSS}. 
\begin{corollary}
  Let the assumptions in Theorem \ref{LSS_result} hold. Then the measures
  ${\mathcal P}_{kn}(\cdot\; |\; m_k)$ converge weakly to the measure 
  ${\mathcal P}(\cdot\; |\; m)$ on $H$.
\end{corollary}
We conclude this section by discussing shortly how to solve reconstructors in practice.
For the moment assume that all the conditions in Theorem \ref{LSS_result} hold 
and $\dim {\rm Ran}(P_k) = K\in \N$.
Moreover, assume $U_n: \Omega \to H_n\subset H$ where $\dim H_n = N\in \N$.
Let ${\mathcal I}_n : H_n \to \R^N$ and ${\mathcal K}_k : {\rm Ran} (P_k) \to \R^K$ be isometries
and let us use them to map the computational model \eqref{compmodel} into a matrix equation.
In the following we use bolded notation for vectors and matrices in Euclidian spaces.
Denote ${\bf U}_n = {\mathcal I}_n U_n = (\bu^N_1, ..., \bu^N_N)^T : \Omega \to \R^N$.
This yields
\begin{equation}
  \label{matrixmodel}
  {\bf M}_{kn} = {\mathcal K}_k M_{kn}= {\bf A}_{kn} {\bf U}_n + {\bf E}_k
\end{equation}
where ${\bf A}_{kn} \in \R^{K\times N}$ and ${\bf M}_{kn},{\bf E}_k : \Omega \to \R^K$.
The posterior density function $\pi_{kn}$ can now be easily obtained for problem \eqref{matrixmodel} via the Bayes formula.
In Section \ref{sec:comp} assumptions on the noise $\Ec$ 
and the measurement projections imply that ${\bf E}_k$ is white noise.
In such a case $\pi_{kn}$ has the form 
\begin{equation*}
  \pi_{kn}(\bu_n \;|\; \bm_k) =
  \frac{\Pi_n(\bu_n) 
    \exp(-\frac 12 \norm{\bm_k - {\bf A}_{kn} \bu_n}^2_2)}{\Upsilon_{kn}(\bm_{kn})},
\end{equation*}
where $\Pi_n$ is the prior density and $\Upsilon_{kn}$ is the density function of ${\bf M}_{kn}$.
For a related discussion on the discretization of white noise see the Appendix B in \cite{LSS}.
The CM estimate corresponds to a reconstructor with $g = {\rm id} : H\to H$ and 
it can be obtained by computing integral
\begin{equation}
  \label{cmest}
  \bu^{CM}_{kn} := \int_{\R^N} \bu \pi_{kn} (\bu \; | \; \bm_k) d\bu
\end{equation}
since with the choice of reconstructors in equation \eqref{recon_formula} it holds that
\begin{equation}
  \label{recon_vs_cm}
  {\mathcal R}_{M_{kn}} (U_n \; | \; m_k) = {\mathcal I}^{-1}_n \left( \bu^{CM}_{kn}\right)
\end{equation}
for any $k,n\in\N$.

%------------------------------------
% CONTINUOUS MODEL
%------------------------------------

\section{The continuous prior model}

In this section we introduce a hierarchical probability distribution in $L^2(\T) \times L^2(\T)$ and
prove that it is well-defined. 
Denote first by $\td$ a perturbed derivation
\begin{equation}
  \label{perturbed_der}
  \td = D + \ep^q P : H^1(\T) \to L^2(\T)
\end{equation}
with some $q>1$ and a projection operator $P f(x) = (\int_\T f(t) dt)\onef(x)$ for $f\in L^1(\T)$ and
$\onef(x)=1$ for every $x\in\T$.
The reason for this perturbation
is that the operator $\td : H^1(\T) \to L^2(\T)$ is invertible.
Also denote $L = \td^{-1} : L^2(\T) \to L^2(\T)$ and a multiplication operator
$\Lambda : L^2(\T) \to \Ll(L^2(\T))$ by
\begin{equation*}
  \Lambda(v) f = (\ep^2+v^2)^{-1} f
\end{equation*}
for any $v,f \in L^2(\T)$. Define operators
\begin{equation}
  \label{covariances}
  C_V = \left(\frac{1}{4\ep}I-\ep \Delta\right)^{-1} \quad {\rm and} \quad
  C_U(v) = L \Lambda(v) L^*
\end{equation}
on $L^2(\T)$ with each $v\in L^2(\T)$ where $L^*$ is the Hilbert-adjoint of $L$. 
It is straightforward to show that both operators ($C_U(v)$ with fixed $v$)
are positive self-adjoint trace class operators. This allows us to define the following
Gaussian measures on $L^2(\T)$ which we use in the construction of the prior probability distribution.
\begin{definition}
  \label{def:prior_measures}
  Let $\nu$ be the Gaussian measure on $L^2(\T)$ centered at value $\onef(x)\equiv 1$ with
  covariance operator $C_V$ and with given  $v\in L^2(\T)$ let $\lambda^v$ be the Gaussian measure
  on $L^2(\T)$ centered at $0$ with covariance operator $C_U(v)$.
\end{definition}
\begin{remark}
\label{remark_one}
Now a possible way to proceed is to define
a probability measure $\lambda$ on $(L^2(\T)\times L^2(\T), {\mathcal B}(L^2(\T)\times L^2(\T)))$ 
in such a way that with any measurable sets $E,F\subset {\mathcal B}(L^2(\T))$ we have
\begin{equation}
  \label{jointprior}
  \lambda(E \times F) = \int_F \lambda^v (E) d\nu(v)
\end{equation}
and assign $\lambda$ as a distribution to a random variable $(U,V) : \Omega \to L^2(\T)\times L^2(\T)$.
In fact, finding a unique extension to $\lambda$ for all Borel sets
connects this problem to more general considerations of 
the existence of Markov chains with given transition
operators \cite{GS2,Fu,BG}. The unique extension can be shown to exist
using results related to stochastic kernels \cite{kallenberg}.
Also, in the framework of $M$-spaces and Markov operators 
the extension result here can be proved using Lemma 1.3 in \cite{Pi}.

However, in the rest of the paper the marginal distributions of $\lambda$
play a central role. 
We achieve more flexible framework especially for the analysis of
the discretized distributions
by constructing a suitable probability space and defining random variables $U$
and $V$ separately. Consequently, we exclude the extension proof at this stage since later
the joint distribution of $(U,V)$ is shown to satisfy 
equation \eqref{jointprior} as a byproduct of the construction.
\end{remark}

\begin{remark}
  \label{remark_two}
  Throughout the rest of the paper we keep $\epsilon>0$ and $q>1$ fixed. 
  The role of $\epsilon$ is to control how sharp edges we will have in the reconstructions.
\end{remark}

To simplify our notations we assume that the probability space has the additional structure
$\Omega_{pr}=\Omega_1\times\Omega_2$, $\Sigma_{pr}=\overline{\Sigma_1\otimes\Sigma_2}$ and
$\Pro_{pr}=\Pro_1\otimes\Pro_2$. 
\begin{definition}
  Let $V : \Omega_2\to L^2(\T)$ be a random variable with distribution $\nu$.
\end{definition}
We note that $V$ has a very similar distribution with the so-called Gaussian smoothness prior.
The smoothness prior is well-known to have realizations in $H^s(\T)$ almost surely
for any $s<1/2$ and this can similarly be shown to $V$.
In fact here the one-dimensional domain allows us to go further with the smoothness.
Below the notation $C^{0,\alpha}$ refers to H\"older spaces with exponent $\alpha>0$ 
and $W^{t,p}$ denotes the $L^p$-based Sobolev space with exponent $t\in\R$ (see \cite{Adams}).

\begin{lemma}
  \label{smoothnessprior}
  The random variable $V : \Omega_2 \to L^2(\T)$ satisfies following two statements:
  \begin{itemize}
  \item[(i)] For any $t<1/2$ and $1<p<\infty$ we have $V \in W^{t,p}(\T)$ almost surely,
  \begin{equation*}
    \expec \norm{V-\onef}^p_{W^{t,p}}<\infty
  \end{equation*}
  and there exists a version $V'$ of $V$ such that $V':\Omega_2\to W^{t,p}(\T)$
  is measurable.
  \item[(ii)] For any $0<\alpha<1/2$ we have $V \in C^{0,\alpha}(\T)$ almost surely and
  \begin{equation*}
    \expec \norm{V-\onef}_{C^{0,\alpha}} < \infty.
  \end{equation*}
  \end{itemize}
\end{lemma}

\begin{proof}
Consider the centered variable $V' = V-\onef$.
By the Schwartz kernel theorem there exists a unique
distribution $K_{V'} \in {\mathcal D}'(\T  \times \T )$ such that 
$\bra C_{V'} \phi,\psi\cet = \bra K_{V'}, \phi \otimes \psi \cet$.
It is straightforward to verify that $K_{V'}$ 
is the Green function of $\frac{1}{4\epsilon}I-\epsilon \Delta$.
Such a function is known to be Lipschitz continuous, i.e., $K_{V'}\in C^{0,1}(\T\times\T)$
and even in $C^\infty$ outside the diagonal.
Let $t\in [0,\frac 12)$ and define a new kernel $K$ on $\T^2$ as
\begin{equation}
  K(x,y)=(1-\Delta_x)^{t/2}(1-\Delta_y)^{t/2}K_{V'}(x,y).
\end{equation}
Now by \cite[Prop. 13.8.3]{T3} and \cite[Sect. 13, (8.7)]{T3}, we have
$K(x,y)\in C^{0,1-2t}(\T \times \T )$
and since $t<\frac 12$, we have in particular that $K$ is continuous and bounded. 
By \cite[Prop. 3.11.15]{B} we have
that for any $1<p<\infty$ there exists a random variable $V_p$ in $L^p(\T)$
with covariance operator $C_p : L^{p'}(\T) \to L^p(\T)$, $\frac 1p + \frac 1{p'} = 1$,
such that
\begin{equation*}
  C_p f(x) = \int_\T K(x,y) f(y) dy.
\end{equation*}
Furthermore, $V_p$ satisfies
\begin{equation*}
  \expec \norm{V_p}^p_{L^p} < \infty.
\end{equation*}
Due to \cite[Prop. 13.8.3]{T3} and \cite[Sect. 13, (8.7)]{T3} we can define for any $1<p<\infty$
a Gaussian centered random variable $V_p'= (I-\Delta)^{-t/2} V_p$ in $W^{t,p}(\T)$
with the property
\begin{equation*}
  \expec \norm{V_p'}^p_{W^{t,p}} < \infty.
\end{equation*}
One notices that the covariance operator of $V_p'$ coincides with $C_{V'}$.
The claim (i) follows from the two distributions being the same.
Furthermore, the Sobolev embedding theorem states that the space $W^{t,p}(\T)$ can be
embedded compactly into $C^{0,t-1/p}(\T)$ \cite{Adams}. This proves the claim (ii).
\end{proof}

\begin{definition}
  \label{vdef}
  From this moment on in all our analysis we replace $V$ with such a version $V'$ 
  that $V'(\omega_2)\in W^{t_0,p_0}(\T)$ for all $\omega\in\Omega$ with some fixed $t_0$ and $p_0$
  and $V':\Omega_2\to W^{t_0,p_0}(\T)$ is measurable. 
  We keep denoting this new random variable by $V$.
\end{definition}

Let $W:\Omega_1\to H^s(\T)$, $s<-1/2$, be a Gaussian random variable satisfying $\expec W = 0$ and
\begin{equation}
  \label{w_expec2}
  \expec (\bra W,\phi\cet_{H^s}\bra W,\psi\cet_{H^s}) 
  = \bra C_s \phi, \psi\cet_{H^s}
\end{equation}
for any $\phi,\psi\in H^s$ where $C_s=(I-\Delta)^s$.
The random variable $W$ is white noise in $H^s(\T)$ in the sense discussed in Section 2. 

In the following the idea is to define $U(\omega_1,\omega_2)$ by operating to $W(\omega_1)$ with
a square root of the mapping $C_U(V(\omega_2))$.
Since $C_U(V(\omega_2))$ was defined above on $L^2(\T)$ we have to be careful how to define the square root.

Let us begin by defining an unbounded bilinear form $b_v: L^2(\T ) \times L^2(\T ) \to \R$,
\begin{equation}
  b_v[\phi,\psi] = \int_{\T } (\ep^2+v^2) \td \phi \cdot \td \psi dx
\end{equation}
for $\phi,\psi\in H^1(\T)$ and $v\in C^{0,\alpha}(\T )$ with $\alpha>0$. 
Due to \cite[Thm. VI.1.21, Thm.VI.2.1]{Kato} there exists a unique linear self-adjoint operator
$B_v:{\mathcal D}(B_v)\to L^2(\T )$, ${\mathcal D}(B_v)=\{\phi\in L^2(\T ) \; 
| \; (\ep^2+v^2)\td \phi \in H^1(\T )\}$, such that
\begin{equation}
  b_v[\phi,\psi] = \bra B_v \phi, \psi \cet
\end{equation}
for all $\phi,\psi \in {\mathcal D}(B_v)$ and ${\mathcal D}(B_v)$ is dense in $L^2(\T )$. Moreover we can deduce
\begin{equation}
  B_v = \td^* (\ep^2+v^2) \td,
\end{equation}
which is an invertible operator from ${\mathcal D}(B_v)$ to $L^2(\T )$. The operator $\td^*$ denotes
the $L^2$-adjoint of $\td$. Clearly, $B_v$ is the inverse of $C_U(v)$ defined in equation 
\eqref{covariances} for any $v\in C^{0,\alpha}(\T)$.

The operator $B_v$ was constructed in such a way that its spectrum in $L^2(\T)$ is strictly positive, i.e.,
$\sigma(B_v) \subset [c,\infty)$ with $c=c(\ep)>0$.
Next let us study the mapping properties of $B_v$ in $H^1(\T)$.
We notice that $B_v : H^1(\T)\to H^{-1}(\T)$ is an invertible mapping and 
the pairing $\bra B_v u, u\cet_{H^{-1}\times H^1}$ can be estimated with the $H^1$-norm of $u$ from
below. For later purposes choose $\delta=\delta(\ep)>0$ such that it satisfies
\begin{equation}
  \label{h1spectrumestimate}
  \bra B_v u, u \cet_{H^{-1}\times H^1} \geq \delta \norm{u}^2_{H^1}
\end{equation}
for $u\in H^1(\T)$.
It is important to note that both $c$ and $\delta$ are independent of $v$.
As the spectrum of $B_v$ is positive we can define 
a square root of $C_U(v)$ as a Dunford-Taylor integral 
\begin{equation}
  \label{root}
  \Gamma_v =\frac 1{2\pi i}\int_\gamma z^{-1/2}(B_v-z)^{-1}\,dz : H^{-1}(\T) \to H^{-1}(\T)
\end{equation}
where $\gamma$ is the curve 
\begin{equation*}
  \gamma=\{z\in \C:\ \dist(z,\R_-)=\tfrac \delta 2 \}
\end{equation*}
oriented in such a way it turns around the origin in the positive direction.
Furthermore, $z\mapsto z^{-1/2}$ maps $\C\setminus \overline \R_-\to \C$
so that $\R_+$ maps to itself. By \cite[Thm. V.3.35]{Kato}
the restriction of $\Gamma_v$ to $L^2(\T)$ is an unbounded self-adjoint operator and 
by \cite[Lemma V.3.36]{Kato} satisfies
\begin{equation}
  \label{sqrt_eq}
  (\Gamma_v|_{L^2})^2 = B_v^{-1}|_{L^2} = C_U(v)
\end{equation}
in $L^2(\T)$.
Next we prove a uniform bound for the norm of $\Gamma_v$.

\begin{lemma}
  \label{bdednessofA}
  There exists a constant $C=C(s,\delta)$
  such that for any $\alpha>0$ and
  for all $v\in C^{0,\alpha}(\T)$ we have
  \begin{equation}
    \label{normestforA}
    \norm{\Gamma_v}_{\Ll(H^{s},L^2)} \leq C
  \end{equation}
  with $s >-1$.
\end{lemma}

\begin{proof}
Let $\alpha>0$ and $v\in C^{0,\alpha}(\T)$. 
We prove the claim by interpolation arguments. First note that
\begin{equation}
  \label{interp 1}
  \norm{(B_v-z)^{-1}}_{\Ll(L^2)} \leq \frac 1{\dist(z,\sigma(B_V))}
\end{equation}
for any $z\in\gamma$. Recall now that $B_v-z$ with $z\in\gamma$ is an invertible operator between
spaces $H^1(\T)$ and $H^{-1}(\T)$. We assume that $f\in H^{-1}(\T)$ and $u\in H^1(\T)$
satisfy equation 
\begin{equation}
  \label{ellipticeq}
  (B_v-z )u=f
\end{equation}
in $H^{-1}(\T)$ for some $z\in\gamma$.
Taking duality pairing of $f$ with $u$ in equation \eqref{ellipticeq} yields then
\begin{equation}
  \label{ellipticeq_adjusted}
  \bra B_v u,u\cet_{H^{-1}\times H^1}= z\norm{u}_{L^2}^2+\bra f,u\cet_{H^{-1}\times H^1}.
\end{equation}
For $z\in \gamma$ we have ${\rm Re }(z)<\delta/2$ and thus
\begin{equation}
  \label{eqdivi}
  \bra B_v u,u\cet_{H^{-1}\times H^1} \leq \frac \delta 2 \norm{u}_{L^2}^2
  +\hbox{Re}\,\bra f,u\cet_{H^{-1}\times H^1}.
\end{equation}
Combining inequalities \eqref{eqdivi} and \eqref{h1spectrumestimate} we get
\begin{equation*}
  \delta\norm{u}_{H^1}^2 \leq \frac {\delta}2\norm{u}_{H^1}^2 + \norm{u}_{H^1}\norm{f}_{H^{-1}}.
\end{equation*}
This yields the bound
\begin{equation}
  \label{boundforb}
  \norm{(z-B_v)^{-1}}_{\Ll(H^{-1},H^1)}\leq \frac 2{\delta}
\end{equation}
when $z\in \gamma$. The equation \eqref{ellipticeq_adjusted} implies
\begin{equation}
  \hbox{Re}\,(-z+\delta)\norm{u}_{L^2}^2=
  -(\bra B_v u,u\cet-\delta\norm{u}_{L^2}^2)+\hbox{Re}\, \bra f,u\cet_{H^{-1}\times H^1}\\ 
\end{equation}
where we have added the term $\delta \norm{u}^2_{L^2}$ and taken the real part.
Again due to inequality \eqref{h1spectrumestimate} the right hand side is less than
$\hbox{Re}\,\bra f,u\cet_{H^{-1}\times H^1}$.
Furthermore by applying the Cauchy-Schwarz inequality and inequality \eqref{boundforb} we have 
\begin{equation}
  \norm{u}_{L^2}^2\leq \frac 1{\hbox{Re}\,(-z+\delta)} \frac 2{\delta}\norm{f}_{H^{-1}}^2 
\end{equation}
which proves the estimate 
\begin{equation}
  \label{interp 2}
  \norm{(z-B_v)^{-1}}_{\Ll(H^{-1},L^2)}\preceq |z|^{-1/2}
\end{equation}
with $z\in\gamma$.
Now we are ready to interpolate (see, e.g., \cite[Prop. 13.6.2]{T3},
\cite{BL} and \cite{Tr}) equations \eqref{interp 1} and \eqref{interp 2}
and get
\begin{equation}
  \label{h-sl2bound}
  \norm{(z-B_v)^{-1}}_{\Ll(H^{s},L^2)}\preceq \left(|z|^{-1/2} \right)^{-s}\left(
  \frac 1{\dist(z,\sigma(B_v))} \right)^{1+s} \preceq |z|^{-1-\frac s2}
\end{equation}
for $-1\leq s \leq 0$.
For $s>-1$ and $z\in \gamma$ we see that
\begin{equation*}
  z^{-1/2}\norm{(z-B_v)^{-1}}_{\Ll(H^s,L^2)}\preceq
  |z|^{-\frac 32 -\frac s2}
\end{equation*}
is an integrable function on $\gamma$. Finally this yields
\begin{equation*}
  \norm{\Gamma_v}_{\Ll(H^s,L^2)}\leq C,
\end{equation*}
for any $s>-1$ with some $C=C(s,\delta)>0$ that is independent of $v$.  
\end{proof}

\begin{definition}
\label{udef}
Define the mapping $U : \Omega \to L^2(\T)$ as
\begin{equation}
  \label{udefeq}
  U(\omega_1,\omega_2) = \Gamma_{V(\omega_2)} W(\omega_1)
\end{equation}
where $W$ is the centered Gaussian random variable 
defined by equation \eqref{w_expec2} in $H^{s}(\T)$ with some $-1<s<-1/2$.
\end{definition}
Let us show that this mapping is measurable and hence a random variable.
Recall that a function $X : \Omega \to H$ is said to be strongly measurable
if there exists a sequence $\{X_j\}_{j=1}^\infty$ of simple functions converging pointwise
to $X$. In separable spaces such as $H^s(\T)$, $s\geq 0$, the measurability is
equivalent to the strong measurability. In addition, an operator valued function 
$X:\Omega \to \Ll(H_1,H_2)$ is said to be strongly measurable if the vector valued
function $\omega\mapsto X(\omega) f$ is strongly measurable in $H_2$ in the sense presented above for all $f\in H_1$.

\begin{proposition}
  \label{prop_av}
  The mapping $\omega_2 \mapsto \Gamma_{V(\omega_2)} \in \Ll(H^s(\T),L^2(\T))$
  is strongly measurable for all $-1<s<-\frac 12$.
\end{proposition}

\begin{proof}
Recall from Definition \ref{vdef} that $V$ is
a $W^{t_0,p_0}$-valued random variable. As such a space is separable we have a sequence
of simple random variables $V_j$ converging pointwise to $V$. 
Due to the Sobolev embedding theorem there exists 
$0<\alpha<1/2$ and $C>0$ such that
\begin{equation}
  \label{simpleapproximation}
  \norm{V_j(\omega_2)-V(\omega_2)}_{C^{0,\alpha}} \leq C
  \norm{V_j(\omega_2)-V(\omega_2)}_{W^{t_0,p_0}}
\end{equation}
for all $\omega_2\in \Omega_2$. Hence $V_j$ converges pointwise also in $C^{0,\alpha}(\T)$.
Next fix $\omega_2\in\Omega_2$ and set $v_j=V_j(\omega_2)$ for all $j\in\N$ and $v=V(\omega_2)$.
Let us factorize the operator
\begin{equation*}
  (B_v-z)^{-1}-(B_{v_j}-z)^{-1}=(B_v-z)^{-1}(B_{v_j}-B_v)(B_{v_j}-z)^{-1} : H^s(\T) \to L^2(\T)
\end{equation*}
where the right hand side operators are considered as a sequence of mappings
\begin{equation*}
  \label{schematicmap}
  H^{s}(\T) \xrightarrow{(B_{v_j}-z)^{-1}} H^1(\T ) 
  \xrightarrow{B_{v_j}-B_v} H^{-1}(\T) \xrightarrow{(B_v-z)^{-1}} L^2(\T).
\end{equation*}
An operator and its adjoint have the same norms and, since 
$\{z \; | \; z \in \gamma\} = \{\bar{z} \; | \; z\in\gamma\}$, inequality \eqref{h-sl2bound} yields
\begin{equation}
  \label{ineq0}
  \norm{(B_v-z)^{-1}}_{\Ll(L^2, H^1)} =
  \norm{(B_v-\bar{z})^{-1}}_{\Ll(H^{-1}, L^2)}
  \preceq |z|^{-1/2}.
\end{equation}
Interpolating inequalities \eqref{ineq0} and \eqref{boundforb} gives us
\begin{equation}
  \label{ineq1}
  \norm{(B_v-z)^{-1}}_{\Ll(H^{s}, H^1)} \preceq 
  |z|^{-\frac 12(1+s)}
\end{equation}
for $-1<s<-1/2$.
In the same way as above we see how the operator $B_v-B_{v_j}$ maps
\begin{equation*}
  H^1(\T ) \xrightarrow{\td} L^2(\T ) \xrightarrow{(v^2-v_j^2)Id} L^2(\T ) 
  \xrightarrow{\td'} H^{-1}(\T ).
\end{equation*}
In this framework the operators $\td$ and $\td'$ are both bounded. 
The multiplication operator is also bounded
and converges to zero in the norm topology due to 
\eqref{simpleapproximation} as $j$ increases. 
Altogether this yields
\begin{equation}
  \label{ineq2}
  \lim_{j\to\infty}\norm{B_{v_j}-B_v}_{\Ll(H^1,H^{-1})} = 0.
\end{equation}
Now returning to random variables $V_j$ and $V$ and adding up inequality
\eqref{ineq1} with \eqref{interp 2} we get
\begin{multline}
  \norm{(B_{V(\omega_2)}-z)^{-1}-(B_{V_j(\omega_2)}-z)^{-1}}_{\Ll(H^s,L^2)} \nonumber \\
  \leq C \norm{B_{V(\omega_2)}-B_{V_j(\omega_2)}}_{\Ll(H^1,H^{-1})} |z|^{-1+\frac s2}
  \label{ineqjotain}
\end{multline}
for all $\omega_2\in\Omega_2$ and furthermore
\begin{multline*}
  \norm{(\Gamma_{V(\omega_2)}- \Gamma_{V_j(\omega_2)})f}_{L^2} \\
  \leq C\int_\gamma |z|^{-\frac 32+\frac s2}
  \norm{B_{V(\omega_2)}-B_{V_j(\omega_2)}}_{\Ll(H^1,H^{-1})}\norm{f}_{H^s} dz \\
  \leq C\norm{B_{V(\omega_2)}-B_{V_j(\omega_2)}}_{\Ll(H^1,H^{-1})} \norm{f}_{H^s}
\end{multline*}
for all $f\in H^s(\T)$ and $\omega_2\in\Omega_2$. Due to equation \eqref{ineq2} this proves the claim.
\end{proof}

\begin{corollary}
  The mapping $U : \Omega \to L^2(\T)$ in Definition \ref{udef} is strongly measurable.
\end{corollary}

\begin{proof}
According to the Proposition \ref{prop_av} we can take simple random variables $\Gamma_{V_j}$ 
that converge pointwise to $\Gamma_V$ in ${\mathcal L}(H^{s}(\T),L^2(\T))$ and simple random variables $W_j$ that
converge pointwise to $W$ in $H^{s}(\T)$ with $s<-1/2$. Now for any $\omega=(\omega_1,\omega_2)\in \Omega$
we have that
\begin{multline*}
  \norm{\Gamma_{V(\omega_2)} W(\omega_1)-\Gamma_{V_j(\omega_2)}W_j(\omega_1)}_{L^2}
  \leq
  \norm{\Gamma_{V(\omega_2)}}_{\Ll(H^s,L^2)} \norm{W(\omega_1)-W_j(\omega_1)}_{H^s} \\
  + \norm{\Gamma_{V(\omega_2)}-\Gamma_{V_j(\omega_2)}}_{\Ll(H^s,L^2)} \norm{W_j(\omega_1)}_{H^s}
\end{multline*}
converges to zero for $-1<s<-1/2$.
\end{proof}

Let us return to the discussion in Remark \ref{remark_one}. 
Also let $-1<s<-1/2$ and fix $\omega_2\in\Omega_2$ and $v=V(\omega_2)$. 
For any $\phi,\psi\in L^2(\T)$ we have
\begin{eqnarray*}
  \expec \bra U(\cdot,\omega_2), \phi\cet_{L^2} \bra U(\cdot,\omega_2),\psi\cet_{L^2}
  & = & \expec \bra W(\cdot), \Gamma_v'\phi\cet_{H^s\times H^{-s}} 
  \bra W(\cdot),\Gamma_v'\psi\cet_{H^s\times H^{-s}} \\
  & = & \expec \bra W(\cdot), C_{-s} \Gamma_v' \phi\cet_{H^{s}}
  \bra W(\cdot), C_{-s} \Gamma_v' \psi\cet_{H^{s}} \\
  & = & \bra C_s C_{-s} \Gamma_v' \phi, C_{-s} \Gamma_v' \psi\cet_{H^{s}}\\
  & = & \bra \Gamma_v^2 \phi, \psi \cet_{L^2} \\
  & = & \bra C_U(v) \phi,\psi\cet_{L^2}
\end{eqnarray*}
where $C_t = (I-\Delta)^t$ for $t\in\R$.
By the Fubini theorem we can deduce that the probability distribution of $(U,V)$ on $L^2(\T)\times L^2(\T)$ is some extension of $\lambda$ defined in equation \eqref{jointprior}.

%---------------------------------------
% DISKREETIT PRIORIT
%---------------------------------------

\section{The finite-dimensional prior model}

We have two objectives in the construction of a finite-dimensional prior model 
for the discretized problem \eqref{compmodel}.
Obviously it is necessary to have weakly converging probability measures. 
After defining $U_n$ and $V_n$ this property is proved later in Section 5.
The second objective is to be able to compute the probability densities explicitly.
For anyone applying such a method in practice it is valuable that no additional approximations are needed.
The main difficulty in obtaining the explicit form is clearly the nonlinear
dependence of $C_U(V)$ with $V$.

The following definitions can be intuitively considered as truncated random series or projections
of the original random variables $U$ and $V$. There is a well-known result
\cite[Prop. 3.5.1]{B} about Gaussian series which states that Cameron-Martin
space provides a natural framework for the basis of the series. Also, as we will see, this
approach makes it easier to control the nonlinearity discussed above.

Notice that the Cameron-Martin spaces $H(\nu)$ and $H(\lambda^{V(\omega_2)})$ 
for all fixed $\omega_2\in\Omega_2$ 
have equivalent norms with the standard norm of $H^1(\T)$. More precisely, the norms satisfy
\begin{equation}
  \label{cm_norms}
  \norm{\cdot}^2_{H(\nu)} = \frac{1}{4\ep}\norm{\cdot}^2_{L^2}+ \ep \norm{D\cdot}^2_{L^2} \quad {\rm and} \quad
  \norm{\cdot}^2_{H(\lambda^v)} = \bra (\ep^2+v^2) \td \cdot, \td \cdot\cet_{L^2}.
\end{equation}
This can be shown by density arguments after the equalities are first established for functions in $C^\infty(\T)$.

Inspired by this connection we show that the continuous and piecewise linear functions provide a suitable
framework for the discretizations. For any $n\in\N$ define
\begin{equation}
  PL(n) = \{f\in C(\T) \; | \; f \textrm{ is linear on each } 
  K^N_j, j=1,...,N \} \subset H^1(\T)
\end{equation}
with $K^N_j=[(j-1)/N,j/N)$, $j=1,...,N$. The value of $N$ depends on $n$ and for 
the rest of the paper we fix notation
\begin{equation*}
  N = N(n) = 2^n.
\end{equation*}
In addition, whenever needed we
consider $\T$ as the closed interval $[0,1]$ with the point $1$ identified as $0$.
Notice that with the notation above $PL(n) \subset PL(n+1)$ for all $n\in\N$. Define also piecewise constant functions
on the same mesh
\begin{equation}
  PC(n) = \{f\in L^2(\T) \; | \; f \textrm{ is constant on each } 
  K^N_j, j=1,...,N \} \subset L^2(\T).
\end{equation}
In the following we use frequently the fact that $\td |_{PL(n)} : PL(n) \to PC(n)$ is an invertible mapping.

\subsection{The definition of $V_n$}

Let us consider for a while $H^1(\T)$ equipped with the inner product $\bra\cdot,\cdot\cet_{H(\nu)}$.
Form an orthonormal basis $\{g_j\}_{j=1}^\infty$ with respect to this inner product
so that for each $n\in \N$ the set $\{g_j\}_{j=1}^N$ spans $PL(n)$. 
Define an orthogonal projection $R_n :H^1(\T)\to PL(n)\subset H^1(\T)$ as
\begin{equation*}
  R_n g = \sum_{j=1}^N \bra g,g_j\cet_{H(\nu)}g_j
\end{equation*}
with $g\in H^1(\T)$. A short computation yields that the corresponding adjoint operator in $H^{-1}(\T)$
is
\begin{equation*}
  R_n'g' = \sum_{j=1}^N \bra g',g_j\cet_{H^{-1}\times H^1} C_V^{-1}g_j
\end{equation*}
for any $g'\in H^{-1}(\T)$.
\begin{definition}
  \label{def:vn}
  Define $V_n:\Omega_2\to PL(n)\subset L^2(\T)$ as
  \begin{equation}
    V_n(\omega_2) = \sum_{j=1}^N \bV^N_j(\omega_2) g_j+\onef,
  \end{equation}
  where $\bV^N_j:\Omega_2\to\R$ are independent random variables with standard normal distribution, $\onef(x)\equiv 1$ 
  and $g_j\in PL(n)$ are as chosen above.
  Denote the probability distribution of $V_n$ on $L^2(\T)$ by $\nu_n$.
\end{definition}
Let us shortly consider the covariance operator of $V_n$ in $L^2(\T)$.
Clearly, for any $\phi\in L^2(\T)$ it holds that
\begin{equation*}
  C_V R_n' \phi = R_n C_V \phi.
\end{equation*}
Furthermore, we have that
\begin{eqnarray*}
  \bra C_{V_n} \phi, \psi\cet_{L^2} & = &
  \expec \bra V_n-\onef, \phi\cet_{L^2} \bra V_n-\onef, \psi\cet_{L^2} \\
  & = & \sum_{j,k=1}^N (\expec \bV^N_j \bV^N_k) \bra g_j,\phi\cet_{L^2} \bra g_k,\psi\cet_{L^2} \\
  & = & \bra \sum_{j=1}^N \bra g_j,\phi\cet_{L^2} g_j,\psi\cet_{L^2} \\
  & = & \bra R_n C_V \phi, \psi\cet_{L^2}
\end{eqnarray*}
for any $\phi,\psi\in L^2(\T)$.
Hence we can conclude that
\begin{equation*}
  C_{V_n} = R_n C_V R_n'|_{L^2} : L^2(\T) \to L^2(\T)
\end{equation*}
for all $n\in\N$.

\subsection{The definition of $U_n$}

The discretization method applied to $V$ cannot be used with $U$ since we do not want 
the corresponding basis to depend on realizations of $V$.
To avoid this consider now $H^1(\T)$ equipped with the inner product
\begin{equation*}
  \bra f,g\cet_{\td} = \bra \td f,\td g\cet_{L^2}
\end{equation*}
for $f,g\in H^1(\T)$.
In the same manner as above form an orthonormal basis 
$\{f_j\}_{j=1}^\infty \subset H^1(\T)$ with respect to 
$\bra\cdot, \cdot\cet_{\td}$ so that for each $n\in\N$ the set $\{f_j\}_{j=1}^N$
spans $PL(n)$. Define then an orthogonal projection $S_n: H^1(\T) \to PL(n) \subset H^1(\T)$ as
\begin{equation}
  \label{def:sn}
  S_n f = \sum_{j=1}^N \bra f, f_j \cet_{\td} f_j
\end{equation}
for any $f\in H^1(\T)$. The dual operator $S_n': H^{-1}\to H^{-1}$ can then be written
\begin{equation*}
  S_n' f' = \sum_{j=1}^N \bra f',f_j\cet_{H^{-1}\times H^1} \td' \td f_j
\end{equation*}
for any $f'\in H^{-1}(\T)$.

The functions $\{\td f_j\}_{j=1}^\infty \subset L^2(\T)$ form by definition an
orthonormal basis to $L^2(\T)$ with respect to the usual inner product of $L^2(\T)$.
Denote by $T_n$ the orthogonal projection 
\begin{equation*}
  T_n \phi = \sum_{j=1}^N \bra \phi, \td f_j \cet_{L^2} \td f_j
\end{equation*}
from $L^2(\T)$ to $PC(n)\subset L^2(\T)$. One notices that 
\begin{equation*}
  \td S_n \td^{-1} \phi = \td \sum_{j=1}^N \bra \td^{-1} \phi, f_j \cet_{\td} f_j = T_n \phi
\end{equation*}
for any $\phi\in L^2(\T)$. The projection $T_n$ is self-adjoint on $L^2(\T)$ and hence
we also have equality $T_n \phi = (\td')^{-1} S_n' \td' \phi$ for any $\phi\in L^2(\T)$.
Let us next show an auxiliary lemma about the convergence of the projections $S_n$.
\begin{lemma}
  \label{uniformproj}
  For the orthogonal projection $S_n$ defined in equation \eqref{def:sn}
  it holds that 
  \begin{equation*}
    \lim_{n\to\infty} \norm{I-S_n}_{\Ll(H^1,H^t)} = 0
  \end{equation*}
  for any $t<\frac 12$.
\end{lemma}
\begin{proof}
Let $t<1/2$ and notice that $(\td \td')^{t-1}$
is a trace class operator in $L^2(\T)$.
Since trace is invariant with respect to the basis and norms $\norm{\cdot}_{H^t}$ and
$\norm{\td^t \cdot}_{L^2}$ are equivalent, we have that
\begin{equation*}
  \sum_{j\in\N} \norm{f_j}^2_{H^t} \preceq \sum_{j\in\N} \norm{\td^t f_j}^2_{L^2}
  = \sum_{j\in\N} \bra \td f_j, (\td \td')^{t-1} \td f_j\cet_{L^2} < \infty
\end{equation*}
since functions $\{\td f_j\}_{j=1}^\infty$ are an orthonormal basis in $L^2(\T)$.
Let $\delta>0$ and choose $N$ so that
\begin{equation*}
  \sum_{j>N} \norm{f_j}^2_{H^t} < \delta.
\end{equation*}
Obviously the functions $\{\td' f_j\}_{j=1}^\infty$ also form 
an orthonormal basis for $L^2(\T)$ and we can write for each $f\in H^1(\T)$
\begin{eqnarray*}
  \norm{f-S_n f}^2_{H^t} & \leq & C\sum_{j=1}^\infty \bra \td^t (f-S_n f), \td' f_j\cet^2_{L^2} \\
  & = & C\sum_{j=1}^\infty \bra \td f, (\td')^{-1} (I-S_n)' \td' (\td')^{t} f_j\cet^2_{L^2} \\
  & \leq & C\sum_{j=1}^\infty \norm{\td f}^2_{L^2} \norm{(I-T_n) (\td')^tf_j}^2_{L^2}
\end{eqnarray*}
since $(\td')^t f_j \in L^2(\T)$.
Hence we can estimate the sum as follows
\begin{eqnarray*}
  \norm{f-S_n f}^2_{H^t} & \leq &
  C\l(\sum_{j=1}^\infty \norm{\sum_{k>n} 
    \bra (\td')^t f_j, \td f_k\cet_{L^2} \td f_k}_{L^2}^2\r) \norm{\td f}^2_{L^2}\\
  & = & C\l(\sum_{j=1}^\infty \sum_{k>n} \bra (\td')^t f_j, \td f_k\cet_{L^2}^2 \r) \norm{\td f}^2_{L^2} \\
  & = & C\l(\sum_{k>n} \norm{(\td')^t f_k}^2_{L^2} \r) \norm{\td f}^2_{L^2} \\
  & \leq & C\delta \norm{f}_{H^1}^2
\end{eqnarray*}
when $n>N$.
\end{proof}

Before defining $U_n$ let us still introduce one more notation. Let $\Lambda_n$ be the 
multiplication operator
\begin{equation*}
  \Lambda_n(v) f = (\epsilon^2+(Q_n v)^2)^{-1} f
\end{equation*}
for any $v\in L^2(\T)$ and $f\in L^2(\T)$
where  
\begin{equation*}
  Q_n v = N \sum_{j=1}^N \int_{K^N_j} v(x) dx \cdot \onef_{K^N_j}
\end{equation*}
and $\onef_{K^N_j}$ is the indicator function of the set $K^N_j = [(j-1)/N,j/N)$.
\begin{definition}
\label{def:un}
Let $U_n:\Omega \to L^2(\T)$ be the random variable 
\begin{equation}
  U_n(\omega_1,\omega_2) = \sum_{j=1}^N \bU^N_j(\omega_1, \omega_2) f_j
\end{equation}
where the random vector $\bU^N(\omega) = (\bU^N_j(\omega))_{j=1}^N \in \R^N$ 
is given the following structure: Denote by $\omega_2\mapsto {\bf C}(\omega_2) \in \R^{N\times N}$ 
a random matrix such that
\begin{equation*}
  {\bf C}_{jk}(\omega_2) = \bra \Lambda_n(V_n(\omega_2)) \td f_j, \td f_k\cet_{L^2}.
\end{equation*}
Due to the positive definiteness of ${\bf C}$ we can define
\begin{equation*}
  \bU^N(\omega) = {\bf C}(\omega_2)^{\frac 12} {\bf W}_N(\omega_1)
\end{equation*}
where ${\bf W}_N : \Omega_1 \to \R^N$ is centered Gaussian random variable with identity covariance matrix.
\end{definition}
The measurability of $\bU^N:\Omega\to \R^N$ is a consequence of the mapping $\omega_2\mapsto V_n(\omega_2)$
being measurable.
Also it follows from Definition \ref{def:un} that with fixed $\omega_2$ the probability distribution of $\omega_1\mapsto U_n(\omega_1,\omega_2)$ is centered Gaussian with
covariance operator
\begin{equation*}
  C_{U_n}(V_n(\omega_2)) = S_n \td^{-1} \Lambda_n (V_n(\omega_2)) (\td')^{-1} (S_n)'|_{L^2(\T)}.
\end{equation*}
This can be seen from the short computation
\begin{eqnarray*}
  \bra C_{U_n} \phi, \psi\cet_{L^2} & = &
  \sum_{j=1}^N \sum_{k=1}^N {\bf C}_{jk} \bra f_j,\phi\cet_{L^2} \bra f_k, \psi\cet_{L^2} \\
  & = & \left\langle \Lambda_n(V_n(\omega_2)) \l(\sum_{j=1}^N \bra f_j,\phi\cet_{L^2} \td f_j\r),
  \sum_{k=1}^N \bra f_k,\psi\cet_{L^2} \td f_k\right\rangle_{L^2} \\
  & = & \bra \Lambda_n(V_n(\omega_2)) (\td')^{-1} (S_n)' \phi, (\td')^{-1} (S_n)' \psi \cet_{L^2}
\end{eqnarray*}
for all $\psi,\phi\in L^2(\T)$.
Denote the distribution of $U_n(\cdot,\omega_2)$ on $L^2(\T)$ by $\lambda^{V_n(\omega_2)}_n$
and the joint distribution of $(U_n,V_n)$ on $L^2(\T)\times L^2(\T)$ by $\lambda_n$.

\subsection{Prior density}

Let us show in this subsection how the prior density function of the random variable $(U_n,V_n)$ can be written down explicitly.
Consider mappings ${\mathcal I}_n, {\mathcal J}_n : PL(n) \to \R^N$ such that
\begin{equation*}
  {\mathcal I}_n \l(\sum_{j=1}^N {\bf x}_j f_j\r) = {\bf x} \quad {\rm and} \quad
  {\mathcal J}_n \l(\sum_{j=1}^N {\bf x}_j g_j\r) = {\bf x}.
\end{equation*}
for any ${\bf x} = ({\bf x}_1, ..., {\bf x}_N)^T\in \R^N$. Use the following notation for the density functions:
let $\Pi_{(\bU_n,\bV_n)}$, $\Pi_{\bV_n}$ and $\Pi_{\bU_n|\bV_n}(\cdot \;|\; {\mathcal J}_nv)$ 
denote the densities of the probability measures $\lambda_n \circ ({\mathcal I}_n^{-1}, {\mathcal J}_n^{-1})$ on $\R^{2N}$ and
$\nu_n \circ {\mathcal J}^{-1}_n$ and $\lambda_n^{v}\circ {\mathcal I}^{-1}_n$ on $\R^{N}$, respectively, with any $v\in PL(n)$. Below $\phi \propto \psi$ denotes relation $\phi \equiv c \psi$ with 
some constant $c$.

\begin{theorem}
  \label{vndensity}
  Let $v\in PL(n)$ be arbitrary and $\bv = {\mathcal J}_n v \in \R^N$. Then
  \begin{equation}
    \Pi_{\bV^N}(\bv) \propto
    \exp\l(-\frac 12 \l(\ep\norm{Dv}^2_{L^2}+
      \frac{1}{4\ep}\norm{v-\onef}^2_{L^2}\r)\r)
  \end{equation}  
  with $\onef(x) = 1$ for all $x\in \T$.
\end{theorem}
\begin{proof}
We recall that by definition $\bV^N_j$ are independent standard Gaussian random 
variables for all $1\leq j \leq N$. It is easy to see that
\begin{equation*}
  \norm{\bv-{\mathcal J}_n\onef}_{\R^N} = 
  \norm{{\mathcal J}_n(v-\onef)}_{R^N} =
  \norm{v-\onef}_{H(\nu)}
\end{equation*}
since ${\mathcal J}_n$ is an isometry between $PL(n)\subset H(\nu)$ and $\R^N$.
By equation \eqref{cm_norms} we now obtain the claim.
\end{proof}

\begin{theorem}
  \label{undensity}
  Let $u, v\in PL(n)$ be arbitrary and $\bu={\mathcal I}_n u, \bv = {\mathcal J}_n v \in \R^N$. Then
  it holds that
  \begin{equation*}
    \Pi_{\bU_n | \bV_n}(\bu\; | \; \bv) \propto
    \exp \l(-\frac 12 \l(\int_\T -N\log(\ep^2+(Q_n v)^2) + (\ep^2+(Q_n v)^2)|\td u|^2 dx\r)\r).
  \end{equation*}
\end{theorem}
\begin{proof}
The density function of a Gaussian random variable in $\R^N$ can be written as
\begin{equation*}
  \Pi_{\bU_n | \bV_n}(\bu\; | \; \bv) \propto \exp \l(-\frac 12 (\log \det {\bf C} + \bra \bu , {\bf C}^{-1} \bu \cet_{\R^N})\r)
\end{equation*}
where the matrix ${\bf C}$ depends on $v$ and its elements satisfy
\begin{equation*}
  {\bf C}_{jk}= \bra \Lambda_n(v) \td f_j, \td f_k\cet_{L^2}
\end{equation*}
for $1\leq j,k\leq N$.
Our challenge is to compute explicitly $\det {\bf C}$ and the inverse matrix ${\bf C}^{-1}$.
Notice first how $\Lambda_n(v)$ maps $PC(n)$ to itself. Inspired by this let us consider
${\bf C}$ as a matrix representation of the linear operator $\Lambda_n(v):PC(n)\to PC(n)$ in the basis $\{\td f_k\}_{k=1}^N$.
Next consider another $L^2$-orthonormal basis for $PC(n)$, namely, 
$\{\sqrt{N} \onef_{K^N_j}\}_{j=1}^N$.
Let the matrix ${\bf S}\in \R^{N\times N}$ be the matrix presentation of the change of the basis
$\{\td f_j\}_{j=1}^N$ to $\{\sqrt{N} \onef_{K^N_j}\}_{j=1}^N$. The components of this matrix are given by the formula
\begin{equation}
  \label{sjk}
  {\bf S}_{jk} = \bra \td f_k, \sqrt{N} \onef_{K^N_j} \cet_{L^2}
\end{equation}
for $1\leq j,k\leq N$.
Moreover, ${\bf S}$ is invertible and satisfies ${\bf S}^{-1} = {\bf S}^T$.

Now the key notion is that 
since $\Lambda_n(v)$ is diagonal in
the basis $\{\sqrt{N} \onef_{K^N_j}\}_{j=1}^N$, we can factorize matrix ${\bf C}$ as
\begin{equation*}
  {\bf C} = {\bf S}^{-1} {\bf L} {\bf S}
\end{equation*}
where the diagonal matrix ${\bf L}$ is the representation of the multiplication 
operator $\Lambda_n(v)$ in the 
basis $\{\sqrt{N} \onef_{K^N_j}\}_{j=1}^N$.
One can show that the diagonal of the matrix ${\bf L}$ consists of elements $\bra (\epsilon^2+(Q_n v)^2)^{-1}, N \onef_{K_j^N}\cet_{L^2}$ for $1\leq j \leq N$.
This immediately yields that
\begin{equation}
  \det {\bf C} = \det {\bf L} = \prod_{j=1}^N \bra (\epsilon^2+(Q_n v)^2)^{-1}, N\onef_{K_j^N}\cet_{L^2}.
\end{equation}
Now we have
\begin{equation*}
  -\log \det {\bf C} = \sum_{j=1}^N -\log \bra (\epsilon^2+(Q_n v)^2)^{-1}, N\onef_{K_j^N}\cet_{L^2}  
  = \int_\T N\log(\epsilon^2+ (Q_nv)^2) dx,
\end{equation*}
which yields the first part of the density function. Furthermore, a simple computation yields
\begin{equation*}
  \bra \bu, {\bf C}^{-1} \bu\cet_{\R^N}
  = \bra {\bf S} \bu, {\bf L}^{-1} {\bf S} \bu\cet_{\R^N}
  = \sum_{j=1}^N \bra \epsilon^2+(Q_n v)^2, N \onef_{K_j^N}\cet_{L^2} ({\bf S} \bu)^2_j.
\end{equation*}
Assume that $u = \sum_{k=1}^N {\bf u}_k f_k$ and ${\bf u} = ({\bf u}_1,...,{\bf u}_N)^T\in \R^N$. 
Then by the equation \eqref{sjk} it holds that
\begin{equation*}
  ({\bf S} \bu)_j = \sum_{k=1}^N {\bf u}_k \bra \td f_k, \sqrt{N} \onef_{K_j^N}\cet_{L^2}
  = \bra \td u, \sqrt{N} \onef_{K_j^N}\cet_{L^2}
\end{equation*}
and finally
\begin{eqnarray*}
  \bra \bu, {\bf C}^{-1} \bu\cet_{\R^N}
  & = & \sum_{j=1}^N \frac 1N \bra \epsilon^2+(Q_n v)^2, N\onef_{K_j^N}\cet_{L^2}
  \bra \td u, N\onef_{K_j^N}\cet_{L^2}^2 \\
  & = & \int_\T (\epsilon^2+(Q_n v)^2) \abs{\td u}^2 dx,
\end{eqnarray*}
which proves the statement.
\end{proof}

We conclude this section by pointing out that 
\begin{equation}
  \label{priordistribution}
  \Pi_{(\bU_n,\bV_n)}(\bu,\bv) = \Pi_{\bV_n}(\bv) \cdot \Pi_{\bU_n|\bV_n}(\bu \;|\; \bv)
\end{equation}
for any $\bu,\bv\in \R^N$. In consequence, the joint density is obtained from lemmata \ref{vndensity}
and \ref{undensity}.

%-----------------------------------------
% CONVERGENCE RESULTS
%-----------------------------------------
 
\section{Convergence of the CM estimates}

Two previous sections were devoted for the construction of the prior distributions. This is however
only halfway in our search for a scalable reconstruction method. In order to 
show the convergence of conditional mean estimates one also has to consider the 
interplay between likelihoods, prior distributions and the measurement equation. 
We turn our attention to this in the following.

\subsection{General conditions}

Some general conditions 
under which reconstructors converge were given in \cite{LSS}.
We generalize these conditions in Theorem \ref{LSS_result}.
The essential difference is that the finite-dimensional priors are not given by linear projections.
Note that here we consider now a general prior random variable $U:\Omega\to H$ with a real separable Hilbert space $H$.
Let us first prove a version of the Vitali convergence theorem for probability measures satisfying
Definition \ref{unifprior}.
\begin{lemma}
  \label{domconv}
  Assume that $\mu_n$ and $\mu$ are uniformly discretized probability measures on $H$.
  Suppose that $f : H \to \R$ is continuous and $0\leq f(u) \leq \exp(b\norm{u}_H)$ 
  for some constant $b>0$.
  Then we have the convergence
  \begin{equation}
    \lim_{n\to\infty} \int_H f(u) d\mu_n(u) = \int_H f(u) d\mu(u).
  \end{equation}
\end{lemma}

\begin{proof}
Let us first denote ${\mathcal B}_j = \{u\in H \; | \; f(u) > j\}$ and $f_j(u) := \min(f(u),j)$ for any
$u\in H$.
We get an upper bound for the probability of ${\mathcal B}_j$ by
\begin{equation}
  \label{bj_size_est}
  \mu({\mathcal B}_j) \leq \frac 1 j \int_{{\mathcal B}_j} f(u) d\mu(u) \leq 
  \frac 1 j \int_H \exp(b\norm{u}) d\mu(u) \leq \frac{C(b)}{j},
\end{equation}
where $C(b)$ is given in Definition \ref{unifprior}.
Notice how the exactly same bound applies also for $\mu_n({\mathcal B}_j)$.
From equation \eqref{bj_size_est} we can deduce
\begin{eqnarray*}
 \int_H \abs{f-f_j} d\mu & = & \int_{{\mathcal B}_j} \abs{f-f_j}d \mu \\ 
  & \leq & 2 \int_{{\mathcal B}_j} \exp(b\norm{u}_H) d\mu(u) \\
  & \leq & 2 \sqrt{C(2b)} \sqrt{\mu({\mathcal B}_j)} \\
  & \leq & \frac{\widetilde{C}(b)}{\sqrt{j}}
\end{eqnarray*}
where $\widetilde{C}(b) = 2 \sqrt{C(2b)} \sqrt{C(b)}$.
Again the same procedure applies for $\mu_n$ yielding the same upper bound. Notice carefully 
that the bound does not depend on $n$. Now the result follows by approximating
\begin{equation*}
  \abs{\int_H f d\mu_n-\int_H f d\mu} \leq \int_H \abs{f-f_j} d\mu + 
  \int_H \abs{f-f_j} d\mu_n + \abs{\int_H f_j(d\mu-d\mu_n)}
\end{equation*}
and using the weak convergence. Namely, for each $\delta>0$ we can choose $j$ so that
we have $\widetilde{C}(b) / \sqrt{j} \leq \delta/3$. On the other hand for each $j$ there exists $n'$ so
that $\abs{\int_H f_j(d\mu-d\mu_n)} < \delta /3$ for each $n>n'$. This results to
\begin{equation*}
   \abs{\int_H f d\mu_n-\int_H f d\mu} < \delta
\end{equation*}
when $n>n'$.
\end{proof}

Combining Lemma \ref{unifprior} and the formula \eqref{recon_formula} we can now prove Theorem \ref{LSS_result}.
\vspace{0.5cm}

{\em Proof of Theorem \ref{LSS_result}.} 
First, let us consider another measurement model
\begin{equation}
  \Theta_{kn} = A_k U_n + \Ec,
\end{equation}
where the noise is not discretized and is now infinite-dimensional. The reconstructor formula can be used for this equation
giving
\begin{equation}
  {\mathcal R}_{\Theta_{kn}} (g(U_n) | m_k) =
  \frac{\int_H g(u) \Xi(u,m_k) d\lambda_n(u)}{\int_H \Xi(u,m_k) d\lambda_n(u)}
\end{equation}
with
\begin{equation}
  \Xi(u,m_k) = \exp(-\frac 12 \norm{A_ku}^2_{L^2}+ \bra C_\Ec^{-1} A_ku, m_k\cet_{H^{-1}}) \leq \exp(b\norm{u}_H)
\end{equation}
with some $b>0$.
Now Lemma \ref{domconv} yields
\begin{equation}
  \lim_{k,n\to\infty} {\mathcal R}_{\Theta_{kn}} (g(U_n) | m_k)
  = {\mathcal R}_{M}(g(U) | m).
\end{equation}
The claim follows from \cite[Lemma 1]{LSS}.
\qed

\subsection{Weak convergence of the prior distribution}

The Proposition 3.8.12. in  \cite{B} yields the weak convergence of measures $\nu_n$.
\begin{lemma}
  The probability distributions $\nu_n$ converge weakly to $\nu$ 
  on $L^2(\T)$.
\end{lemma}
We want to show that with fixed $\omega_2\in\Omega_2$ 
the distribution $\lambda^{V_n(\omega_2)}_n$ converges weakly
to $\lambda^{V(\omega_2)}$. Since $\lambda^{V_n(\omega_2)}_n$ is not obtained with a straight-forward projection as in the case of $\nu_n$
we recall conditions that are needed in the weak convergence of general Gaussian distributions.
The following lemma is proved in \cite{B} as Example 3.8.15.
\begin{lemma}
  \label{aux_weak}
  A sequence of Gaussian measures $\mu_n$ with means $a_n$ and covariance operators $C_n$
  on a separable Hilbert space $H$ converges weakly to a Gaussian measure $\mu$ with mean $a$ 
  and covariance  operator $C$ if and only if the following conditions are satisfied:
  \begin{enumerate}
    \item[(i)] $\lim_{n\to\infty} \norm{a_n-a}_H = 0$,
    \item[(ii)] $\lim_{n\to\infty} \norm{C_n - C}_{{\Ll}(H)} = 0$ and
    \item[(iii)] $\lim_{n\to\infty} \trace_{H}(C_n) = \trace_{H} (C)$.
  \end{enumerate}
\end{lemma}
Let us prove an auxiliary lemma concerning the convergence of the multiplication operators.
\begin{lemma}
  \label{multiplicationconv}
  Let $v_n\to v$ in $W^{t_0,p_0}(\T)$ as $n\to\infty$.
  Then we have
  \begin{equation*}
    \lim_{n\to\infty}\norm{\Lambda(v)-\Lambda_n(v_n)}_{{\mathcal L}(L^2)} = 0.
  \end{equation*}
\end{lemma}
\begin{proof}
First notice that 
for some $\alpha>0$ we have by the Sobolev embedding theorem 
that $\norm{v-v_n}_{C^{0,\alpha}} \to 0$.
For any continuous $f:\T\to\R$ denote
\begin{equation*}
  \norm{f}_\infty = \sup_{x\in\T} |f(x)|.
\end{equation*}
Let us then compute an upper bound
\begin{multline*}
  \abs{\frac 1{\epsilon^2+v^2}-\frac 1{\epsilon^2+(Q_nv_n)^2}}
  \leq \frac{1}{\epsilon^4} \left(\abs{(Q_n v_n)^2-v_n^2}+\abs{v_n^2-v^2}\right) \\
  \leq \frac{1}{\epsilon^4} \left(2\norm{v_n}_\infty \abs{Q_nv_n-v_n} + 
    (\norm{v_n}_\infty+\norm{v}_\infty) \norm{v_n-v}_\infty\right).
\end{multline*}
Here the term $\abs{Q_nv_n-v_n}$ can be estimated pointwise as
\begin{eqnarray*}
  \abs{N \int_{K^N_j} v_n(y) dy-v_n(x)} 
  & = & N \abs{\int_{K^N_j} (v_n(x)-v_n(y)) dy} \\
  & \leq & N \int_{K^N_j} \frac{\abs{v_n(x)-v_n(y)}}{\abs{x-y}^\alpha} \cdot \abs{x-y}^\alpha dy \\
  & \leq & \frac{1}{N^\alpha} \norm{v_n}_{C^{0,\alpha}}
\end{eqnarray*}
where $x\in K^N_j$ and $K^N_j$ is the half-open interval $[(j-1)/N, j/N)$.
The above yields
\begin{equation*}
  \lim_{n\to\infty}\norm{\frac 1{\epsilon^2+v^2}-\frac 1{\epsilon^2+(Q_nv_n)^2}}_\infty = 0
\end{equation*}
and thus
\begin{equation*}
  \lim_{n\to\infty}\norm{\l(\Lambda(v)-\Lambda_n(v_n)\r) f}^2_{L^2} \leq \lim_{n\to\infty}
  \norm{\frac 1{\epsilon^2+v^2}-\frac 1{\epsilon^2+(Q_nv_n)^2}}_\infty^2 \norm{f}_{L^2}^2=0. 
\end{equation*}
for all $f\in L^2(\T)$.
\end{proof}

\begin{lemma}
  \label{weakconv}
  Assume $v_n\in PL(n)$ and $v_n\to v$ in $W^{t_0,p_0}(\T)$.
  The measure $\lambda^{v_n}_n$ converges weakly
  to $\lambda^v$ on $L^2(\T)$.
\end{lemma}
\begin{proof}
The condition (i) in Lemma \ref{aux_weak} holds as the means stay constant. 
Furthermore, condition (ii)
follows from the suitable convergence of the operators $S_n$ and $\Lambda_n$. 
Since $L^* = (\td')^{-1}|_{L^2}$ we see this from
  \begin{eqnarray*}
    \norm{C_{U_n}(v_n)-C_U(v)}_{{\mathcal L}(L^2)} & = &
    \norm{S_n \td^{-1} \Lambda_n(v_n) (\td')^{-1}S_n'
    	-\td^{-1} \Lambda(v) (\td')^{-1}}_{{\mathcal L}(L^2)} \\
    & \leq & \norm{S_n \td^{-1} \Lambda_n(v_n) (\td')^{-1}}_{\Ll(H^{-1},L^2)} 
    \norm{S_n'-I}_{\Ll(L^2,H^{-1})} \\
    & & + \norm{S_n-I}_{\Ll(H^1,L^2)} \norm{\td^{-1} \Lambda_n(v_n) (\td')^{-1}}_{\Ll(L^2,H^1)} \\
    & & + \norm{\td^{-1} (\Lambda_n(v_n)-\Lambda(v))(\td')^{-1}}_{\Ll(L^2)}.
  \end{eqnarray*}
In the first two terms of the right hand side recall that $\Lambda_n(v_n)$ is uniformly bounded in 
$\Ll(L^2(\T))$, i.e., the bound is independent of 
$v_n$. Since also $\td^{-1}$ is bounded from $L^2(\T)$
to $H^1(\T)$ we see that Lemma \ref{uniformproj} provides the convergence of these terms. The convergence of the third term
follows from Lemma \ref{multiplicationconv}. 

Let us next consider condition (iii). 
Recall now the projection $T_n = \td S_n \td^{-1} :L^2(\T)\to L^2(\T)$. 
In the following we consider $T_n$ from $L^2(\T)$ to $H^s(\T)$, $s<0$, and hence
the dual operators occur.
Denote $e_j(x) = e^{2\pi i j x}$ for all $j\in \Z$ and notice
$\abs{(\td')^{-1} e_j} \preceq \apr{j}^{-1}$ where $\apr{j} = |j|+1$. We can then write 
\begin{eqnarray*}
  \bra (C_U(v) - C_{U_n}(v_n)) e_j,e_j \cet_{L^2} 
  & = & \bra (T_n \Lambda_n(v_n)T_n' -\Lambda(v)) (\td')^{-1} e_j, (\td')^{-1} e_j\cet_{L^2} \\
  & = & \bra ((T_n-I) \Lambda_n(v_n)T_n') (\td')^{-1} e_j, (\td')^{-1} e_j \cet_{L^2}\\
  & & + \bra (\Lambda_n(v_n) - \Lambda(v)) T_n' (\td')^{-1} e_j, (\td')^{-1} e_j\cet_{L^2} \\
  & & + \bra \Lambda(v)(T_n'-I) (\td')^{-1} e_j, (\td')^{-1} e_j\cet_{L^2}.
\end{eqnarray*}
Let us study the three terms separately: 
a dual norm estimation yields an upper bound for the first term
\begin{multline*}
  \bra ((T_n-I) \Lambda_n(v_n)T_n') (\td')^{-1} e_j, (\td')^{-1} e_j \cet_{L^2}\\
  \leq \norm{T_n-I}_{\Ll(L^2,H^{s})} \norm{\Lambda_n(v_n)}_{\Ll(L^2)}
  \norm{T_n'}_{\Ll(L^2)} \norm{(\td')^{-1} e_j}_{L^2} \norm{(\td')^{-1} e_j}_{H^{-s}} \\
  \leq C \norm{T_n-I}_{\Ll(L^2,H^{s})} j^{-2-s}
\end{multline*}
for any $-1<s<-1/2$.
In the second term we can use Lemma \ref{multiplicationconv} to get
\begin{multline*}
  \bra (\Lambda_n(v_n) - \Lambda(v)) T_n' (\td')^{-1} e_j, (\td')^{-1} e_j\cet_{L^2} \\
  \leq \norm{\Lambda_n(v_n) - \Lambda(v)}_{\Ll(L^2)} \norm{T_n'}_{\Ll(L^2)} \norm{(\td')^{-1} e_j}_{L^2}^2 
  \leq o(n) j^{-2},
\end{multline*}
where $o :\N\to [0,\infty)$ denotes a function that satisfies $\lim_{n\to\infty} o(n)=0$.
The third term yields similar upper estimate as the first term since
\begin{multline*}
  \bra \Lambda(v)(T_n'-I) (\td')^{-1} e_j, (\td')^{-1} e_j\cet_{L^2} \\
  \leq \norm{\Lambda(v)}_{\Ll(L^2)} \norm{T_n'-I}_{\Ll(H^{-s},L^2)} 
  \norm{(\td')^{-1} e_j}_{H^{-s}} \norm{(\td')^{-1} e_j}_{L^2} \\
  \leq C\norm{T_n-I}_{\Ll(L^2,H^{s})} j^{-2-s}.
\end{multline*}
Due to Lemma \ref{uniformproj} and the fact that $\td$ is invertible between $H^t(\T)$ 
and $H^{t-1}(\T)$ for any $t\in\R$ we have $\norm{T_n-I}_{\Ll(L^2,H^{s})} = o(n)$. 
Combining these three bounds yields
\begin{equation*}
  \bra (T_n \Lambda_n(v)T_n' -\Lambda(v)) (\td')^{-1} e_j, (\td')^{-1} e_j\cet_{L^2}
  \leq o(n) j^{-2-s}.
\end{equation*}
Since $\sum_{j=1}^\infty j^{-2-s}$ with $-s>1$ is finite, 
we have shown that ${\rm Tr}_{L^2}(C_U(v) - C_{U_n}(v))$ converges to zero.
This concludes the proof.
\end{proof}

Let us recall the Skorohod coupling theorem.

\begin{theorem}
  \label{skorohod}
  Suppose that a sequence of Borel probability measures $\mu_n$ on a 
  complete separable metric space $B$ converges weakly to
  a Borel measure $\mu$. Then there exists a probability space 
  $(\Omega, \Pro)$ and measurable mappings $X,X_n : \Omega \to B$
  such that $\mu_n = \Pro\circ X_n^{-1}$, $\mu = \Pro \circ X^{-1}$ and $X_n \to X$ a.s.
\end{theorem}

At this point we fix $\Omega_2$ according to Theorem \ref{skorohod} in such a way
that $V_n \to V$ in $W^{t_0,p_0}(\T)$ almost surely. This choice is made
to achieve the final result. Before following theorem recall the definition of uniform tightness:
A sequence $\{\mu_n\}_{n=1}^\infty$ Borel measures on Banach space $X$ 
is said to be uniformly tight if for every $\delta>0$ there exists a compact set $K_\delta\subset X$
such that $\mu_n(X\setminus K_\delta) < \delta$ for every $n\in\N$.

% WEAK CONVERGENCE - IMPORTANT FIX

\begin{theorem}
  When $n$ goes to infinity the random variable $(U_n,V_n)$ converges in distribution to $(U,V)$ in $L^2(\T)\times L^2(\T)$.
\end{theorem}
\begin{proof}
Let us first show the uniform tightness of the sequence $\{\lambda_n\}_{n=1}^\infty$ where
$\lambda_n$ is the joint distribution of $(U_n,V_n)$ on $L^2(\T)\times L^2(\T)$.
The convergence of $V_n$ in distribution yields that probability measures $\{\nu_n\}_{n=1}^\infty$
are uniformly tight. Let $\delta>0$ be given and choose a compact set $K_2\subset L^2(\T)$ 
in such a way that $\nu_n(K_2) > 1-\frac{\delta}{2}$. Next we consider the tightness of a family
$\{\lambda_n^v \; | \; v\in K_2,n\in \N\}$. By Lemma \ref{weakconv} the sequence 
$\{\lambda_n^{{\bf 0}}\}_{n=1}^\infty$ converges weakly and in consequence is uniformly tight.
We choose $K_1\subset L^2(\T)$ so that $\lambda_n^{{\bf 0}}(K_1)>1-\frac{\delta}{2}$.
We may also assume that $K_1$ is absolutely convex since by Proposition A.1.6 in \cite{B}
closed absolutely convex hulls of compact sets are compact. 
Recall the definition of the covariance
$C_U(v) = L \Lambda(v) L^*$ of $\lambda_n^v$ in equation \eqref{covariances}. 
For any fixed $v\in L^2(\T)$
we know that
\begin{equation*}
  \int_{L^2(\T)} \bra u,\phi\cet^2_{L^2} d\lambda_n^v(u)
  = \bra \Lambda(v) L^* \phi, L^* \phi\cet_{L^2}
  \leq \frac{1}{\epsilon^2} \norm{L^* \phi}^2_{L^2}
  = \int_{L^2(\T)} \bra u,\phi\cet^2_{L^2} d\lambda_n^{{\bf 0}}(u).
\end{equation*}
for all $\phi\in L^2(\T)$.
By Theorem 3.3.6 of \cite{B} this yields 
\begin{equation*}
1-\frac{\delta}{2}< \lambda_n^{{\bf 0}}(K_1) \leq \lambda_n^{v}(K_1).
\end{equation*}
Now we are able to deduce the uniform tightness of $\{\lambda_n\}$ by setting 
$K_\delta = K_1\times K_2$, namely,
\begin{multline*}
  \lambda_n((L^2(\T) \times L^2(\T)) \setminus K_\delta)  \\ =
  \lambda_n((L^2(\T) \setminus K_1)\times K_2) + \lambda_n(L^2(\T) \times (L^2(\T) \setminus K_2)) \\
  \leq \int_{K_2} \lambda^v_n(L^2(\T) \setminus K_1) d\nu_n(v)
  + \int_{L^2(\T) \setminus K_2} \lambda^v_n(L^2(\T)) d\nu_n(v) \\
  \leq \frac{\delta}{2} + \frac{\delta}{2} = \delta.
\end{multline*}
Moreover, by the Fubini theorem the characteristic function of $(U_n,V_n)$ can be written as
\begin{multline*}
\expec \exp\left(i\bra U_n,\phi\cet_{L^2}+ i \bra V_n, \psi\cet_{L^2}\right) \\
= \expec \left(\int_{L^2(\T)} 
\exp(i\bra u,\phi\cet_{L^2}) d\lambda^{V_n(\omega_2)}_n(u)
\exp(i\bra V_n(\omega_2),\psi\cet_{L^2}) \right).
\end{multline*}
The almost sure convergence of $V_n$ and Lemma \ref{weakconv} together imply
\begin{equation*}
  \lim_{n\to\infty} \int_{L^2(\T)} \exp(i\bra u,\phi\cet_{L^2}) d\lambda^{V_n(\omega_2)}_n(u)
  = \int_{L^2(\T)} \exp(i\bra u,\phi\cet_{L^2}) d\lambda^{V(\omega_2)}(u)
\end{equation*}
and furthermore
\begin{equation*}
  \lim_{n\to\infty} \exp(i\bra V_n(\omega_2),\psi\cet_{L^2}) = \exp(i \bra V(\omega_2),\psi\cet_{L^2})
\end{equation*}
for almost every $\omega_2\in \Omega_2$. In consequence, we see by the Lebesgue dominated
convergence theorem that the characteristic functions of $(U_n,V_n)$ converge to
the characteristic function of $(U,V)$ pointwise.

By Corollary 3.8.5 in \cite{B} the uniform tightness and
pointwise converging characteristic functions yield that the random variables $(U_n,V_n)$ 
converge in distribution. Since two measures on ${\mathcal B}(L^2(\T)\times L^2(\T))$
with equal characteristic functionals coincide we conclude that $(U,V)$ is a limit.
\end{proof}

\subsection{Uniformly finite exponential moments}

In this section we establish the uniform exponential boundedness of $(U_n,V_n)$, $n\in\N$ and
$(U,V)$. Here we denote
\begin{equation*}
  \norm{(f,g)}_{L^2\times L^2} := \sqrt{\norm{f}_{L^2}^2 + \norm{g}^2_{L^2}}
\end{equation*}
for all $f,g\in L^2(\T)$.
\begin{lemma}
  \label{expglu}
  For every $b>0$ there exists a constant $C(b)>0$ such that
  \begin{equation}
    \label{expgrowthbound}
    \expec \exp(b\norm{(U_n,V_n)}_{L^2\times L^2}) < C(b) \quad {\rm and} \quad
    \expec \exp(b\norm{(U,V)}_{L^2\times L^2}) < C(b)
  \end{equation}
  for every $n\in \N$.
\end{lemma}
\begin{proof}
Let us first show the 
boundedness of the exponential moments of $(U,V)$.
By using the inequality $\norm{(f,g)}_{L^2\times L^2} \leq \norm{f}_{L^2} + \norm{g}_{L^2}$ and Lemma \ref{bdednessofA}
we have
\begin{eqnarray}
  \expec \exp(b\norm{(U,V)}_{L^2\times L^2})
  & \leq & \expec \exp(b\norm{\Gamma_{V(\omega_2)}W(\omega_1)}_{L^2}
  + b\norm{V(\omega_2)}_{L^2}) \nonumber \\
  & \leq & \expec \exp(\tilde b\norm{W(\omega_1)}_{H^s})
  \cdot \expec \exp(b\norm{V(\omega_2)}_{L^2})
  \label{expectationfubini}
\end{eqnarray}
with some constant $\tilde b>0$ and some $-1<s<-\frac 12$.
Moreover, the Fernique theorem \cite[Thm. 2.6]{DaPrato} states 
that for every Gaussian random variable $X$ in Banach space $(B,{\mathcal B}(B))$ there exists a constant $a>0$
such that 
\begin{equation*}
  \expec \exp(a\norm{X-\expec X}_B^2) < \infty.
\end{equation*}
Let $b\in\R$ be arbitrary. The trivial estimate $0 \leq a(x-b/2a)^2$ for any $x\in\R$ yields
\begin{equation}
  \label{trickeq}
  \expec \exp(b\norm{X}_B) \leq \exp(b\norm{\expec X}_B) \cdot \exp(b^2/4a) \cdot \expec \exp(a\norm{X-\expec X}_B^2) < \infty.
\end{equation}
Now the claim for $(U,V)$ follows by applying inequality \eqref{trickeq}
to the right-hand side of inequality \eqref{expectationfubini}.

The uniform bound for $(U_n,V_n)$, $n\in\N$, requires more careful analysis.
Consider for the moment a Gaussian random variable $X$ in $L^2(\T)$ with covariance operator $C_X:L^2(\T)\to L^2(\T)$
such that $\dim {\rm Ran}(C_X) = \ell<\infty$.
Let $\{\rho_j\}_{j=1}^\ell$ be the non-zero eigenvalues 
and $\{\phi_j\}_{j=1}^\ell$ be the corresponding $L^2$-normalized eigenvectors of $C_X$.
Notice that the normal random variables $\bra X-\expec X,\phi_j\cet_{L^2}$ and $\bra X-\expec X,\phi_k\cet_{L^2}$
are independent when $j\neq k$. For any $a < 1/(2\rho_j)$, $1\leq j\leq \ell$, we have
\begin{equation*}
  \expec(\exp(a\bra X-\expec X,\phi_j\cet_{L^2}^2)) = (1-2a\rho_j)^{-\frac 12}.
\end{equation*}
The operator $C_X$ is positive definite and hence
\begin{equation*}
  \max_{j\in\{1,...,\ell\}} \rho_j \leq \trace_{L^2}(C_X).
\end{equation*}
Notice now that $(1-s)^{-1/2}<1+s$ with $0<s<1/2$. 
In consequence, if $a$ satisfies
\begin{equation*}
  a < \frac{1}{4 \trace_{L^2} (C_X)}
\end{equation*}
then for every $j=1,...,\ell$ it follows that 
\begin{equation}
  \label{aux_exp_1}
  \expec (\exp(a\bra X-\expec X,\phi_j\cet^2)) \leq 1 + 2a\rho_j \leq\exp(2a\rho_j).
\end{equation}
Due to the independence of random variables $\bra X-\expec X,\phi_j\cet_{L^2}$ 
and \cite{Vak2} we have
\begin{multline}
  \expec \exp(a \norm{X-\expec X}_{L^2}^2)
  = \prod_{j=1}^\ell \expec \exp(a\bra X-\expec X,\phi_j\cet_{L^2}^2) \\
  \leq \exp(2a\sum_{j=1}^\ell \rho_j) \leq \exp(2a\trace_{L^2}(C_X)) < \infty
  \label{aux_ineq_exp_1}
\end{multline}
where we have used the inequality \eqref{aux_exp_1}. Combining inequalities \eqref{trickeq}
and \eqref{aux_ineq_exp_1} in the case $B=L^2(\T)$ yields
\begin{equation}
  \label{aux_ineq_exp_2}
  \expec \exp(b\norm{X}_{L^2}) \leq \exp(b\norm{\expec X}_{L^2}) \cdot \exp(b^2/4a)\cdot\exp(2a\trace_{L^2}(C_X)).
\end{equation}
Let us next show that the trace of $C_{U_n}(V_n(\omega_2))$ is bounded uniformly with respect to $n\in\N$
and $\omega_2\in\Omega_2$. Denote $e_j(x) = \exp(-2\pi ijx)$ for $x\in\T$ and $j\in\Z$.
A straightforward computation yields
\begin{eqnarray}
  \trace_{L^2} (C_{U_n}(V_n(\omega_2))) & = &  
  \sum_{j\in\Z} \bra \Lambda_n(V_n(\omega_2)) T_n \td^{-1} e_j, T_n \td^{-1} e_j\cet_{L^2} \nonumber \\
  & \leq & \frac{1}{\ep^2} \sum_{j\in\Z} \norm{T_n \td^{-1} e_j}_{L^2} \nonumber \\
  & \leq &\frac{1}{\epsilon^2} \sum_{j\in\Z} \norm{\td^{-1} e_j}_{L^2}^2=C'<\infty \label{ctconst}
\end{eqnarray}
for some constant $C'<\infty$. Clearly $C'$ does not depend on $n$ or $\omega_2$.
With similar arguments we can show that
\begin{equation}
  \label{ctconst2}
  \trace_{L^2} (C_{V_n}) \leq C''
\end{equation}
where constant $C''$ does not depend on $n$.
By the Fubini theorem we have
\begin{multline*}
  \expec \exp(b\norm{(U_n,V_n)}_{L^2\times L^2}) \leq \\
  \int_{L^2(\T)} \left(\int_{L^2(\T)} \exp(b \norm{u}_{L^2}) d\lambda^{v}_n(u)\cdot \exp(b\norm{v}_{L^2})\right) d\nu_n(v)
\end{multline*}
and finally due to inequalities \eqref{aux_ineq_exp_2}, \eqref{ctconst} and
\eqref{ctconst2} we obtain
\begin{eqnarray*}
  \expec \exp(b\norm{(U_n,V_n)}_{L^2\times L^2}) 
  &\leq & \exp(b^2/4a) \exp(2a C') \expec\exp(b\norm{V_n(\omega_2)}_{L^2}) \\
  & \leq & \exp(b^2/2a + b + 2a(C'+C''))
\end{eqnarray*}
for any $a< \frac 14 \min(\frac 1{C'},\frac 1{C''})$. The claim follows by taking the maximum of the bounds on $(U,V)$
and $(U_n,V_n)$, $n\in\N$.
\end{proof}

%-----------------------------------------
% NUMERIIKKA JA ESIMERKIT
%-----------------------------------------

\section{Computational example}
\label{sec:comp}

In this section we illustrate by a numerical example how the method produces reconstructions 
with similar properties as Ambrosio-Tortorelli minimization \cite{AT1,AT2} in deterministic case.
We show how the choice of $\epsilon$ controls the edge-preserving
property of our reconstruction method. Moreover, we compute reconstructions 
with different choices of $n$ to convince the reader that the estimates stay stable.

\subsection{The model problem}

Let us consider a Bayesian deblurring problem $M = AU+\Ec$
on $\T$ where $A:L^2(\T)\to C^\infty(\T)$ is the operator
\begin{equation}
  Au(x) = \int_\T K(x,y) u(y) dy
\end{equation}
with a priori known smooth kernel $K$ satisfying $\int_\T K(x',y) dx' = \int_\T K(x,y') dy' = 1$
for all $x,y\in \T$.
Assume also the following two properties:
\begin{itemize}
\item[(i)] the noise $\Ec$ can be modeled by white noise statistics and
\item[(ii)] the measurement projection $P_k : L^2(\T) \to PL(k)$ is proper in the sense of Definition \ref{propermeasurement}.
\end{itemize}
As we have earlier discussed the assumptions above are related to the measurement situation. 
Let us then implement the prior distributions and discretization introduced in previous sections.
Recall mappings ${\mathcal I}_n,{\mathcal J}_n : PL(n) \to \R^N$ with $N=2^n$ from Section 3.3.
Using Theorems \ref{vndensity} and \ref{undensity} we see that the posterior density
for computational model \eqref{compmodel} has 
the following form: let $\bu = {\mathcal I}_n(u)$ and $\bv = {\mathcal J}_n(v)$. Then we have
\begin{equation}
  \pi_{kn}(\bu,\bv \;|\; \bm) \propto \exp (-\frac 12 F_{\epsilon,k,n} (u,v \; | \;m))
\end{equation}
where $\bu,\bv \in \R^N$, $\bm \in \R^K$ and
\begin{eqnarray*}
  F_{\epsilon,k,n}(u,v \;|\; m) & = & \int_\T \left(-N \log(\epsilon^2+(Q_n v)^2) 
  +(\epsilon^2+(Q_n v)^2) \abs{\td u}^2\right. \\
  & &\left.+ \epsilon \abs{D v}^2 + \frac 1{4\epsilon} (1-v)^2 + \abs{A_{kn}u-m}^2 \right)dx
\end{eqnarray*}
where $u,v\in PL(n)$, $m\in PL(k)$, $N=2^n$ and $K=2^k$. 
Due to equation \eqref{recon_vs_cm} the computational task is then to evaluate integrals
\begin{eqnarray}
  \bu^{CM}_{kn} & = & \int_{\R^N\times \R^N} \bu\cdot \pi_{kn} (\bu,\bv \; | \; \bm) \, d\bu d\bv\quad {\rm and} \nonumber \\
  \bv^{CM}_{kn} & = & \int_{\R^N\times \R^N} \bv\cdot \pi_{kn} (\bu,\bv \; | \; \bm) \, d\bu d\bv.\label{comp_cm_est}
\end{eqnarray}

\subsection{Computation of the CM estimates}

The integrals in equation \eqref{comp_cm_est} are taken over a very large dimensional space and for that reason
it is impossible to implement efficiently any quadrature rule. Usually 
in such situations different types of
Markov Chain Monte Carlo (MCMC) methods are used to obtain a solution.
In the following let us ease our presentation by denoting
\begin{equation*}
  \bw = \vek{\bu}{\bv} \in \R^{2N}.
\end{equation*}
The idea of MCMC algorithms is to generate a collection 
$\bw^1,...,\bw^L \in \R^{2N}$ of samples
according to the posterior distribution. When $L$ is large
we can approximate the CM estimates in \eqref{comp_cm_est} by
\begin{equation}
  \vek{\bu^{CM}_{kn}}{\bv^{CM}_{kn}} = \bw^{CM}_{kn} = \int_{\R^{2N}} \bw \cdot \pi_{kn}(\bw \; 
  | \; \bm )\; d\bw \approx \frac{1}{L-\ell_0} \sum_{\ell=\ell_0+1}^L
  \bw^\ell
\end{equation}
where $\ell_0$ stands for the number of samples in a {\it burn-in period}, i.e., the samples
that do not explore the posterior distribution representatively and are discarded.

The algorithm used here for generating the ensemble is an adaptive version of the Met\-ro\-polis--Hastings (MH) algorithm
\cite{Has,Met,lehtinen2,HST1,SV},
namely single component adaptive Metropolis (SCAM)
algorithm introduced in \cite{HST}. The SCAM algorithm is similar to the basic single component Metropolis
algorithm in the sense that a sample state, say, $\bw^\ell$ is attained by updating the coordinates separately.
When deciding the $j^{th}$ coordinate $\bw^\ell_j$ a sample is drawn from the normal distribution
${\mathcal N}(\bw^{\ell-1}_j, \sigma^\ell_j)$ centered at the previous point with variance $\sigma^\ell_j$. The difference
is to update variances $\sigma^\ell_j$ according to the rule 
\begin{equation}
  \label{varupd}
  \sigma^\ell_j = \left\{ 
    \begin{array}{ll}
      \sigma^0_j, & \ell\leq\ell_0, \\
      s{\rm Var}\left(\bw^0_j,\bw^1_j, ..., \bw^{\ell-1}_j\right)+\delta, & \ell>\ell_0.
    \end{array}
  \right.
\end{equation}
Here $s$ denotes the scaling factor for which the value $s=2.4$ (see \cite{HST,Gelman}) is used here.
The role of $\delta$ is to prevent the variance from shrinking to zero and a small constant ($\delta = 10^{-3}$)
is used as its value. We close this section by showing in pseudo-code how the SCAM algorithm can be implemented.
\begin{itemize}
\item[(1)] Initialize $\bw^0\in \R^{2N}$ and variances $(\sigma^0_i)_{i=1}^{2N}$. Set $\ell:=1$ and $j:=1$.

\item[(2)] Update $\sigma^\ell_j$ from formula \eqref{varupd}.

\item[(3)] Sample $\tau_j\in \R$ from ${\mathcal N}(0,\sigma^\ell_j)$ and set
  \begin{eqnarray*}
    \bw^{new} & = & (\bw_1^\ell, ... , \bw^\ell_{j-1}, \bw^{\ell-1}_j+\tau_j,\bw^{\ell-1}_{j+1}, ..., \bw^{\ell-1}_{2N})^T
    \quad {\rm and} \\
    \bw^{old} & = & (\bw_1^\ell, ... , \bw^\ell_{j-1}, \bw^{\ell-1}_j,\bw^{\ell-1}_{j+1}, ..., \bw^{\ell-1}_{2N})^T.
  \end{eqnarray*}

\item[(4)] If 
\begin{equation*}
  \pi_{kn}(\bw^{new} \;|\; \bm) \geq \pi_{kn}(\bw^{old}\; | \; \bm),
\end{equation*}
set $\bw^\ell_j := \bw^{\ell-1}_j + \tau_j$; and go to 6. 
\item[(5)] Draw a random number $t$ from the uniform distribution on $[0,1]$.
If 
\begin{equation*}
  t\leq \frac{\pi_{kn}(\bw^{old} \; | \; \bm)}{\pi_{kn}(\bw^{new}\; | \; \bm)},
\end{equation*}
set $\bw^\ell_j :=\bw^{\ell-1}_j + \tau_j$; else set $\bw^\ell_j := \bw^{\ell-1}_j$.
\item[(6)] If $j<2N$, set $j \leftarrow j+1$ and go to 2; else if $j=2N$ and $\ell<L$,
set $\ell \leftarrow \ell+1$ and $j\leftarrow 1$ and go to 2; else if $j=2N$ and $\ell=L$ then stop.
\end{itemize}

\subsection{Results}

All computations were done using the interval $[0,1]$ with point $1$ identified as $0$.
Here the parameter for measurement nodes is kept fixed and is chosen to be $k=7$, i.e., we have $K = 2^k = 128$
measurement nodes. The number of nodes for the estimates varies between 64 and 256, i.e., $n$ varies between 6 and 8.
See Figure 1 for the exact solution $u \in L^2(\T)$ and the measured data $m_k\in PL(k)$. The noise in the measurement
was produced from a white noise distribution. Parameters of the MCMC computations are given in Table 1; in each case we take
initial values that correspond zero function for $u$ and $\onef(x)\equiv 1$ function for $v$. Both Figures 2 and 3 illustrate
how the results look when $n$ is increased. The difference between the two figures is the choice of $\epsilon$; in Figure 2
we have chosen $\epsilon = 10^{-3}$ and in Figure 3 the corresponding value is $3\times 10^{-4}$. Moreover, parameter $q$ in
\eqref{perturbed_der} was chosen large enough in order to get quantity $\epsilon^q$ neglectable.

We perform all the computations with Matlab 7.6 running in a desktop PC computer with an AMD Opteron 265 dual-dual 
processor and 8 GB of RAM. Note that the algorithm is not parallelized and thus only one of the processors 
running at 1,8GHz was in full use at a time.

\begin{table}
\caption{Parameters of MCMC computations. The number $N$ is the dimension of reconstruction, 
$\epsilon$ is the prior parameter, $L-\ell_0$
is the number of samples used for computing the CM estimate, $r$ is the total acceptance ratio, i.e., all
samples accepted vs. samples tested and the last column indicates the amount of CPU time used for computations.}
\begin{center}
\begin{tabular}{ccccc}
\hline
$N$ & $\epsilon$ & $L-\ell_0$ & $r$ & Time (h)\\
\hline
64 & $10^{-3}$ & $10^6$ & 0.35 & 6.6 \\
64 & $3\times 10^{-4}$ & $10^6$ & 0.36 & 7.3 \\
128 & $10^{-3}$ & $2\times 10^6$ & 0.27 & 25.3 \\
128 & $3\times 10^{-4}$ & $2\times 10^6$ & 0.33 & 26.9 \\
256 & $10^{-3}$ & $2\times 10^6$ & 0.18 & 50.6 \\
256 & $3\times 10^{-4}$ & $2\times 10^6$ & 0.25 & 53.7 \\
\hline
\end{tabular}
\end{center}
\label{table1}
\end{table}

\begin{figure}[h]
\begin{picture}(380,100)(40,10)
\epsfxsize=5.8cm
\epsfysize=3.1cm
\put(40,0){\epsffile{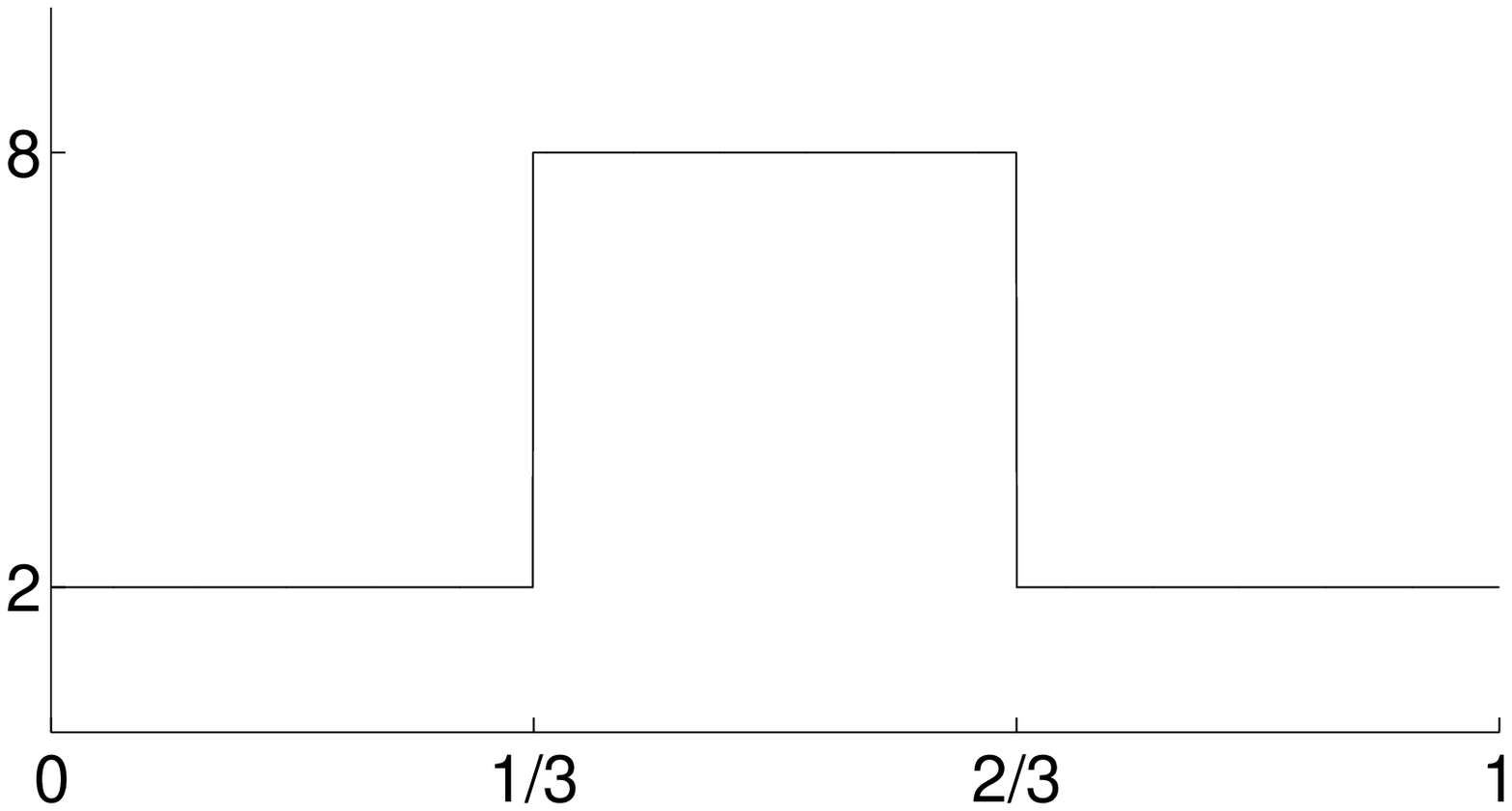}}
\put(230,0){\epsffile{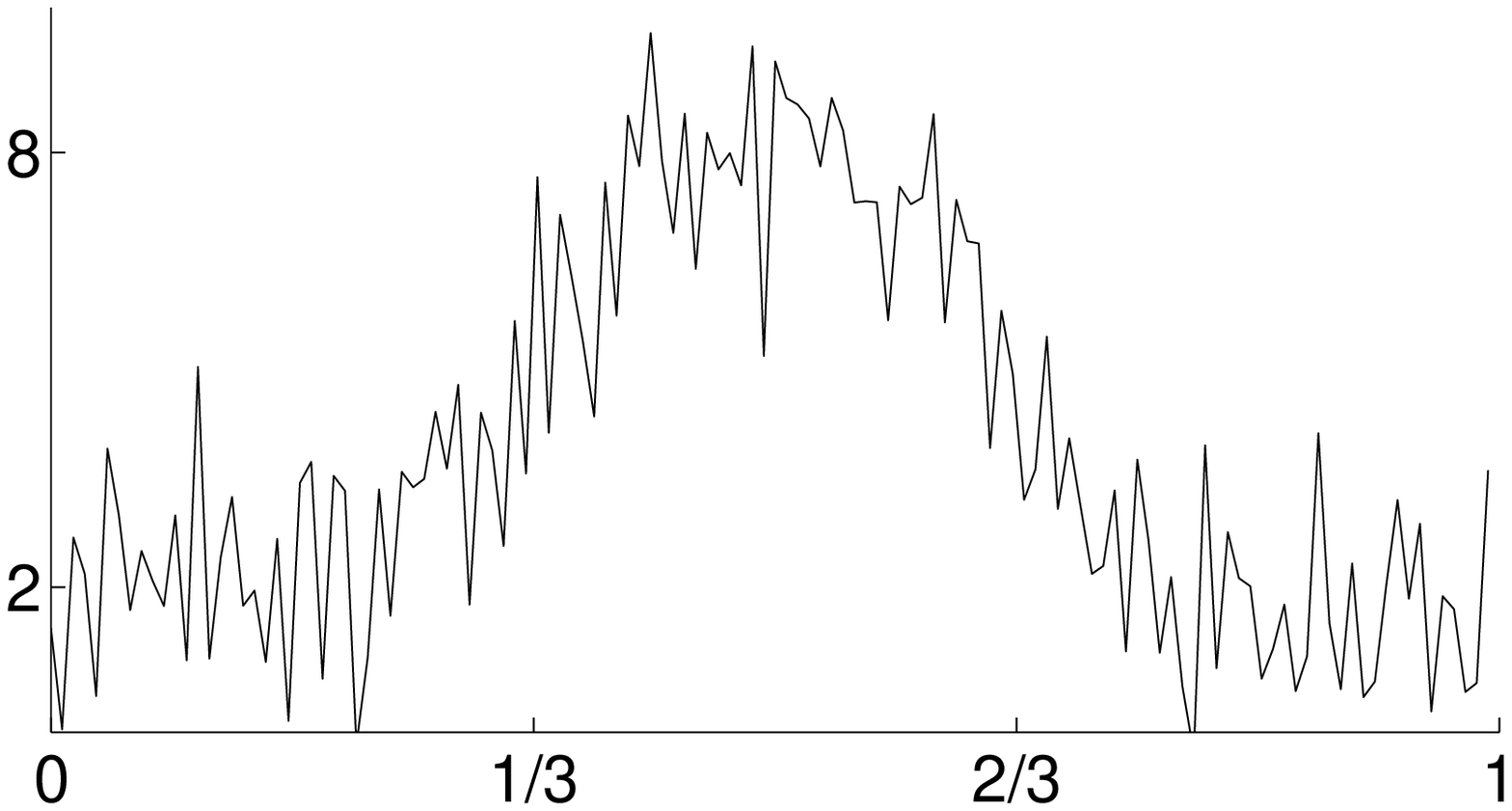}}
\end{picture}
\caption{Left: exact solution $u$, Right: measurement $m_k=M_k(\omega_0)$.}
\label{fig_exact}
\end{figure}

\begin{figure}[h]
\begin{picture}(400,180)(40,10)

\epsfxsize=3.8cm 
\epsfysize=2.5cm
\put(40,90){\epsffile{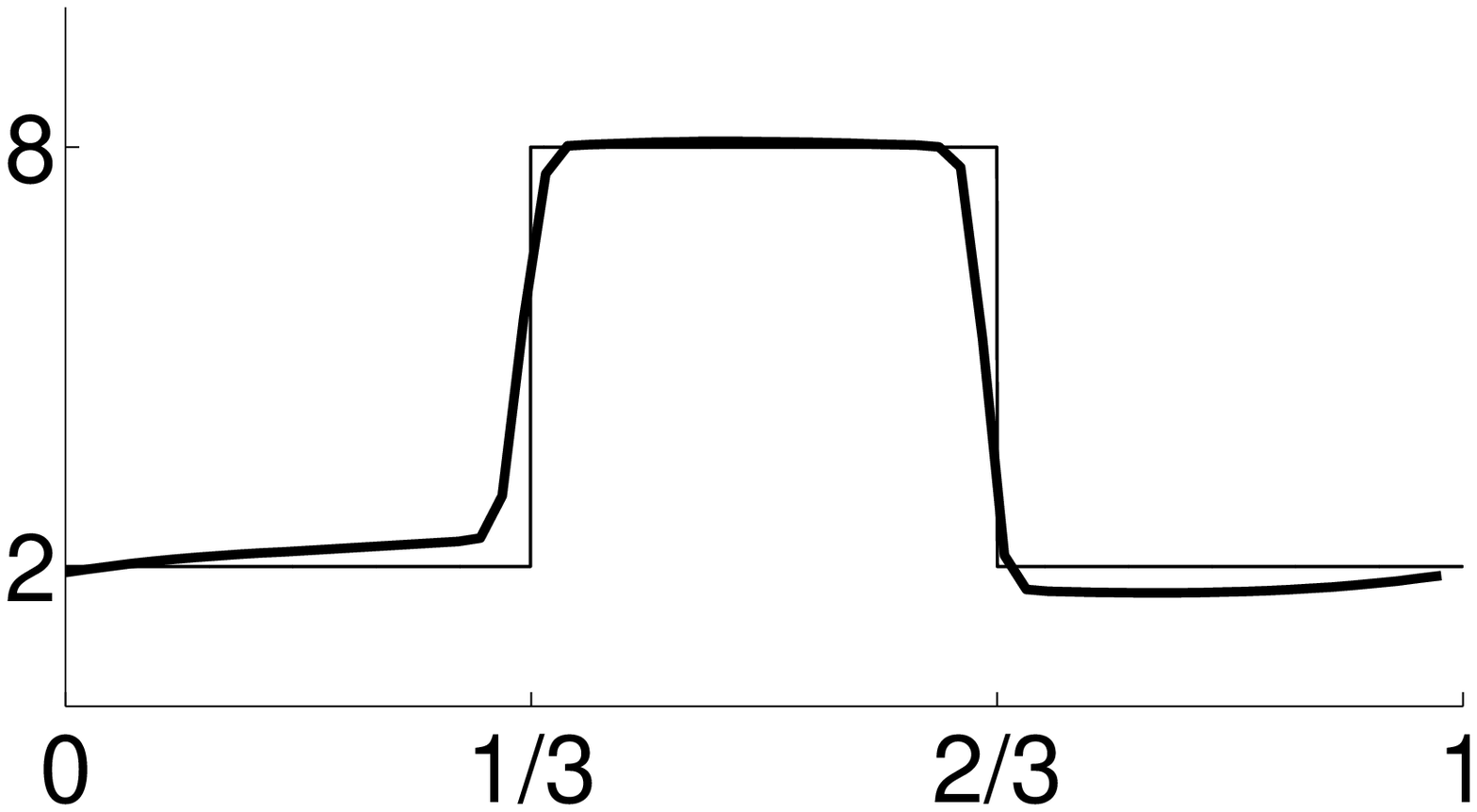}}
\put(40,165){$N=64$}
\put(165,90){\epsffile{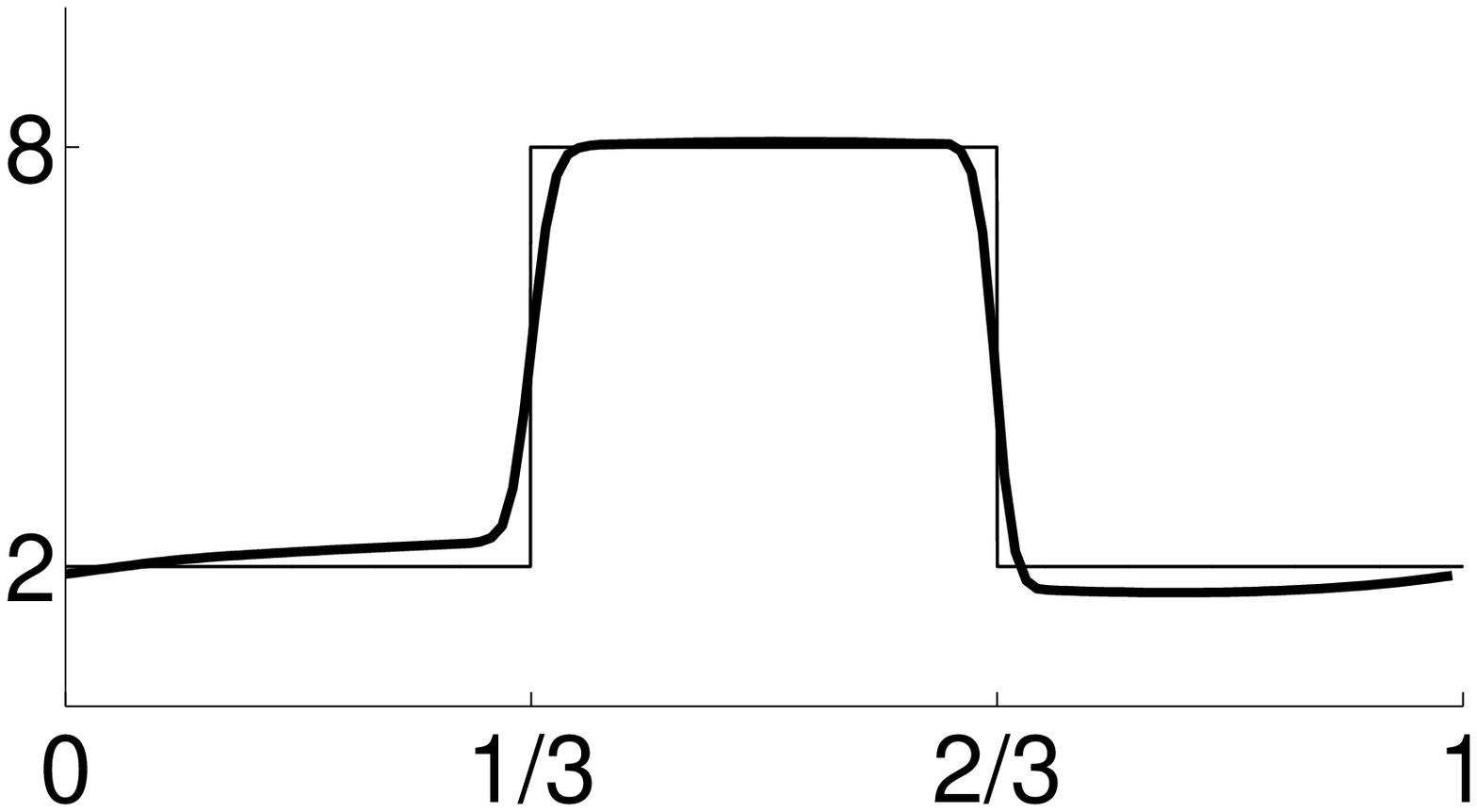}}
\put(165,165){$N=128$}
\put(290,90){\epsffile{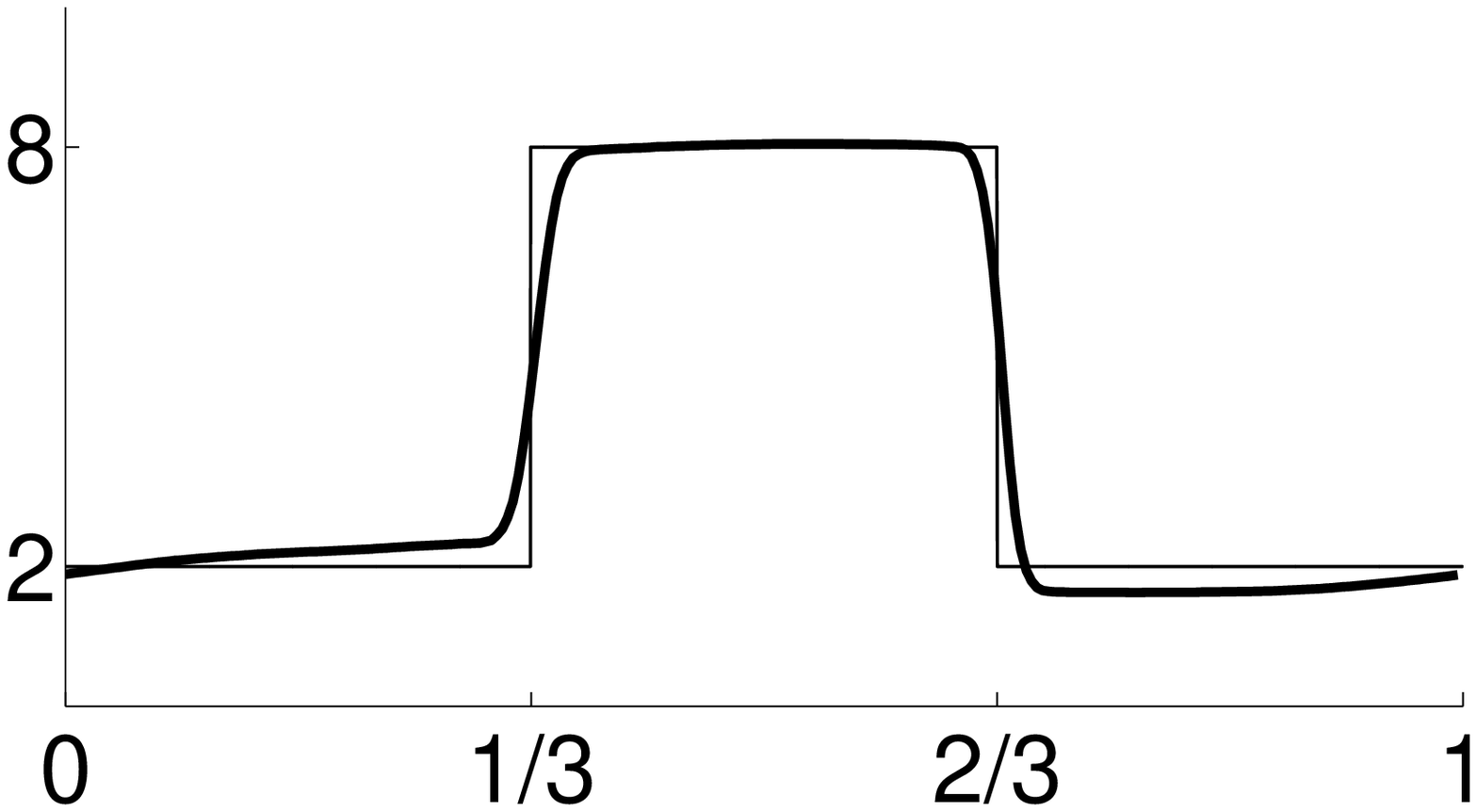}} 
\put(290,165){$N=256$}

\put(40,10){\epsffile{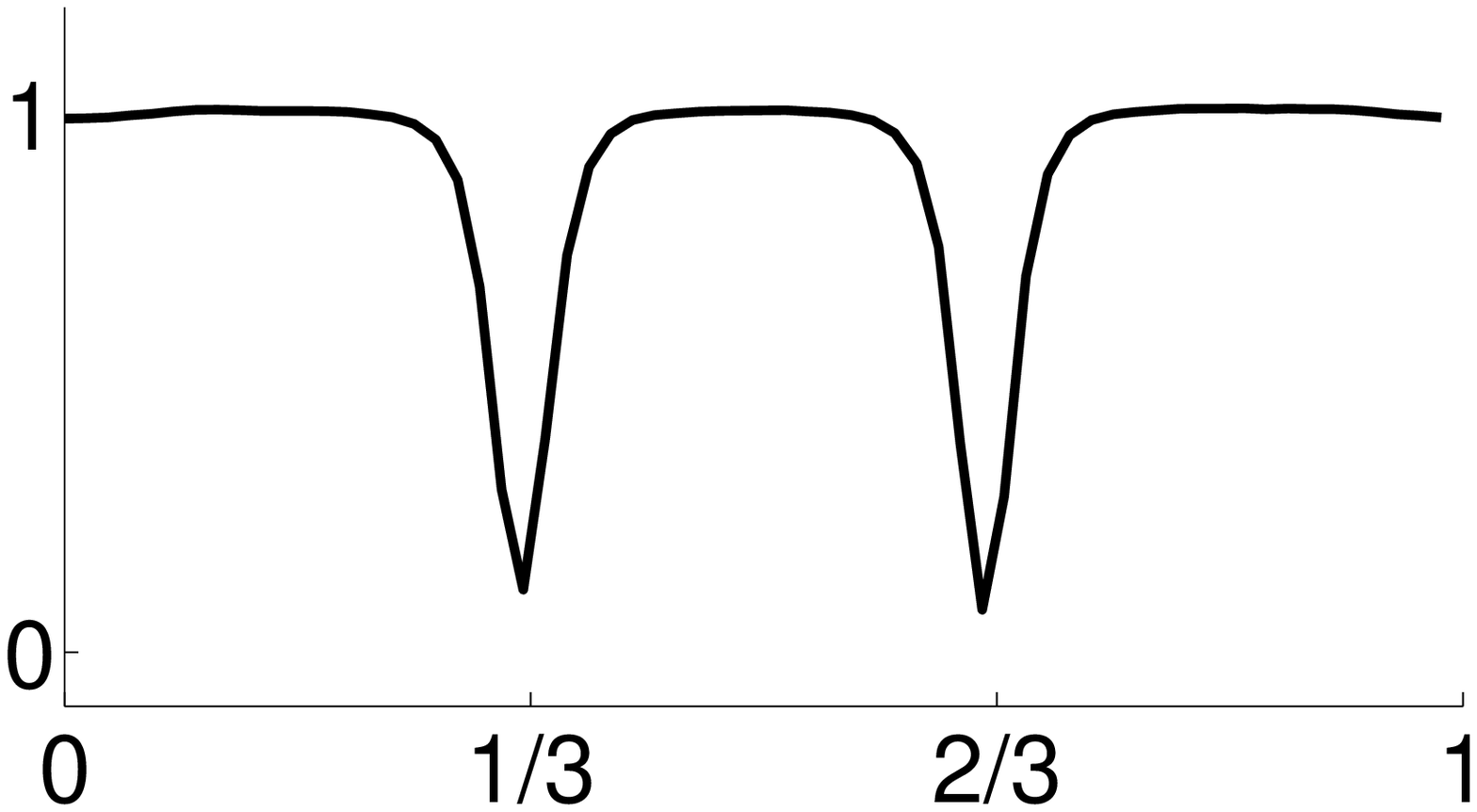}} 
\put(165,10){\epsffile{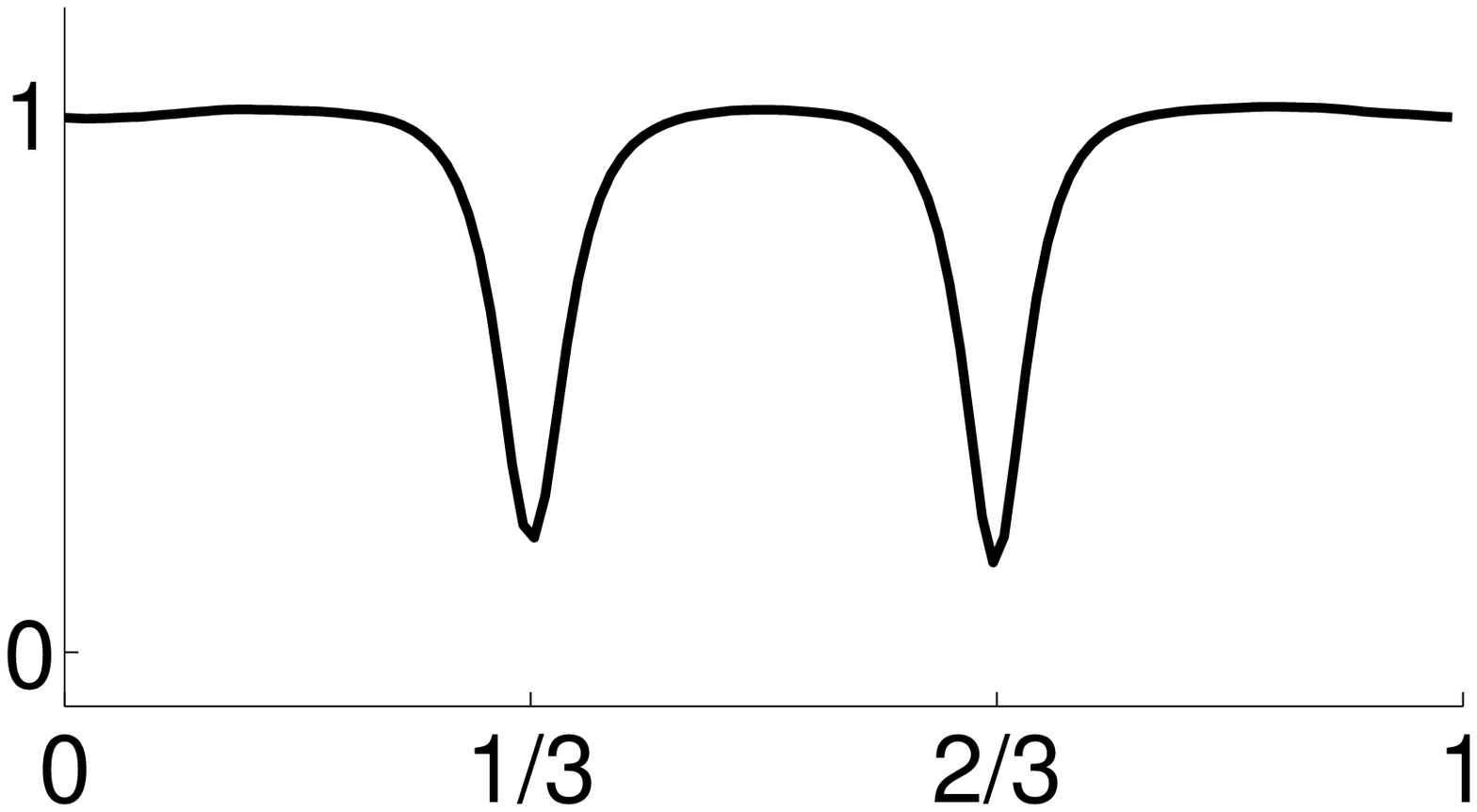}}
\put(290,10){\epsffile{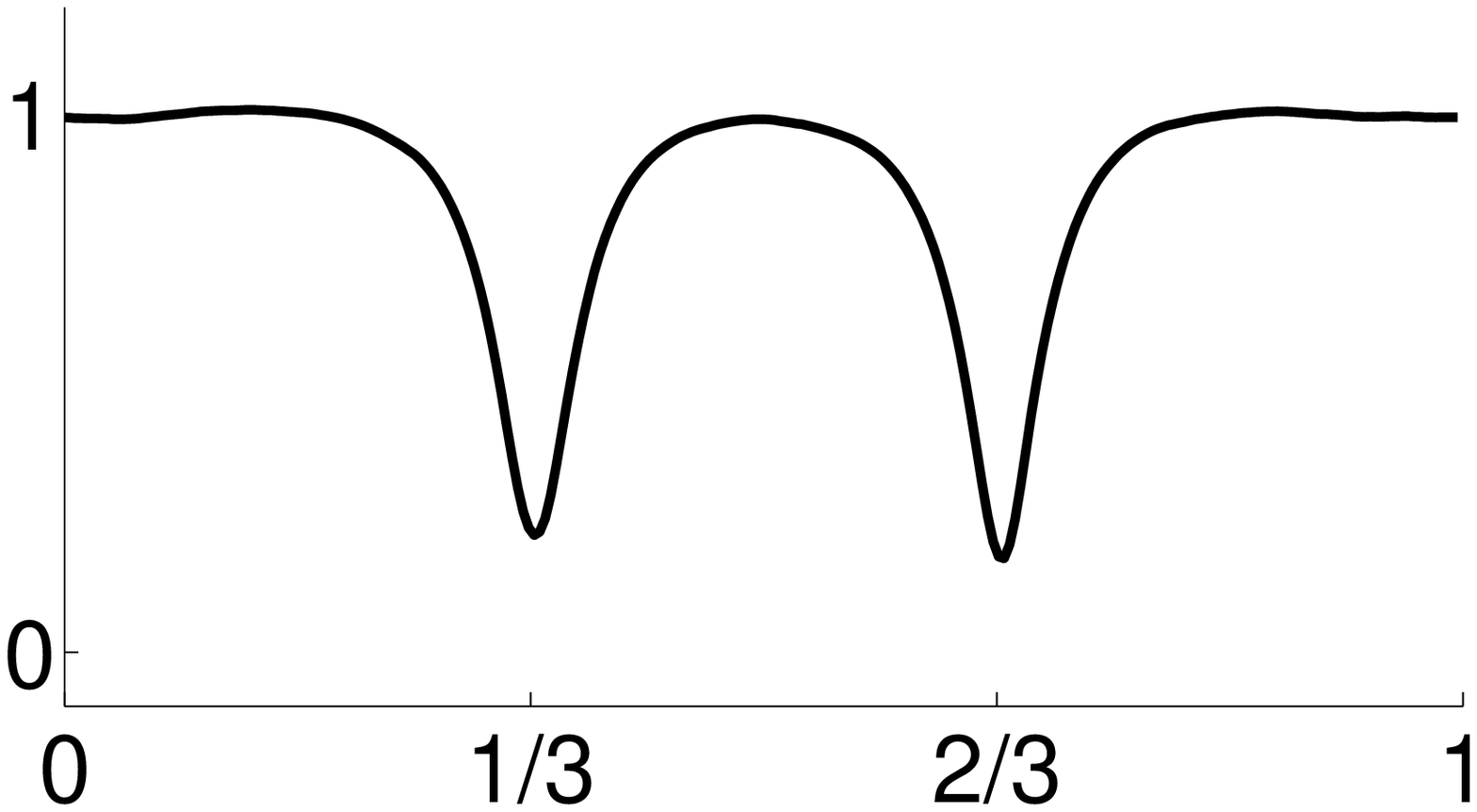}} 

\end{picture}
\caption{All the plots in this figure are obtained with the choice $\epsilon = 10^{-3}$ and $k=7$. 
Top row: the CM estimates $u^{CM}_{kn}$ with $n=6,7,8$ (thick line) and the true signal (thin line)
Bottom row: the CM estimates $v^{CM}_{kn}$.}
\label{fig2}
\end{figure}

\begin{figure}[h]
\begin{picture}(400,180)(40,10)

\epsfxsize=3.8cm 
\epsfysize=2.5cm
\put(40,90){\epsffile{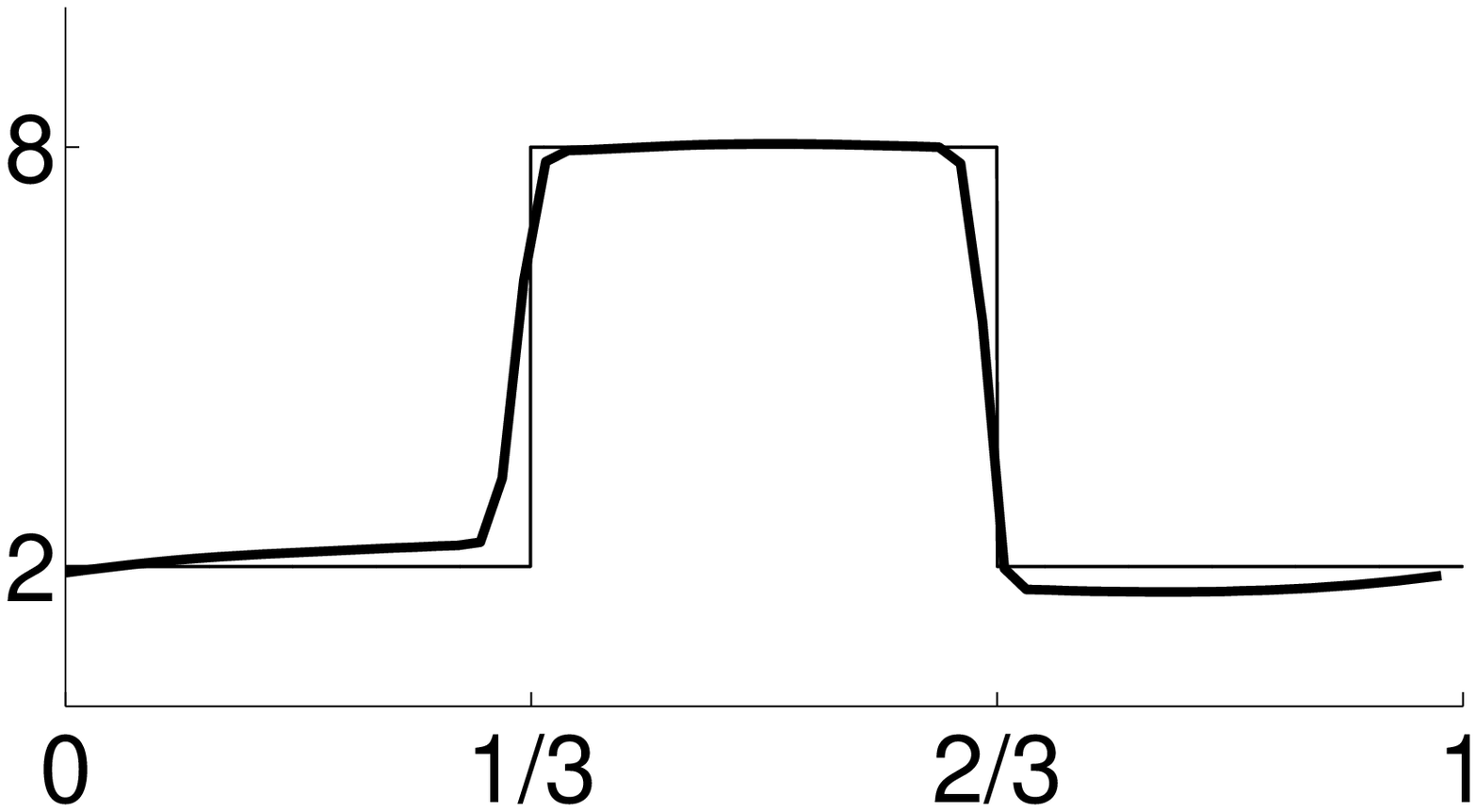}}
\put(40,165){$N=64$}
\put(165,90){\epsffile{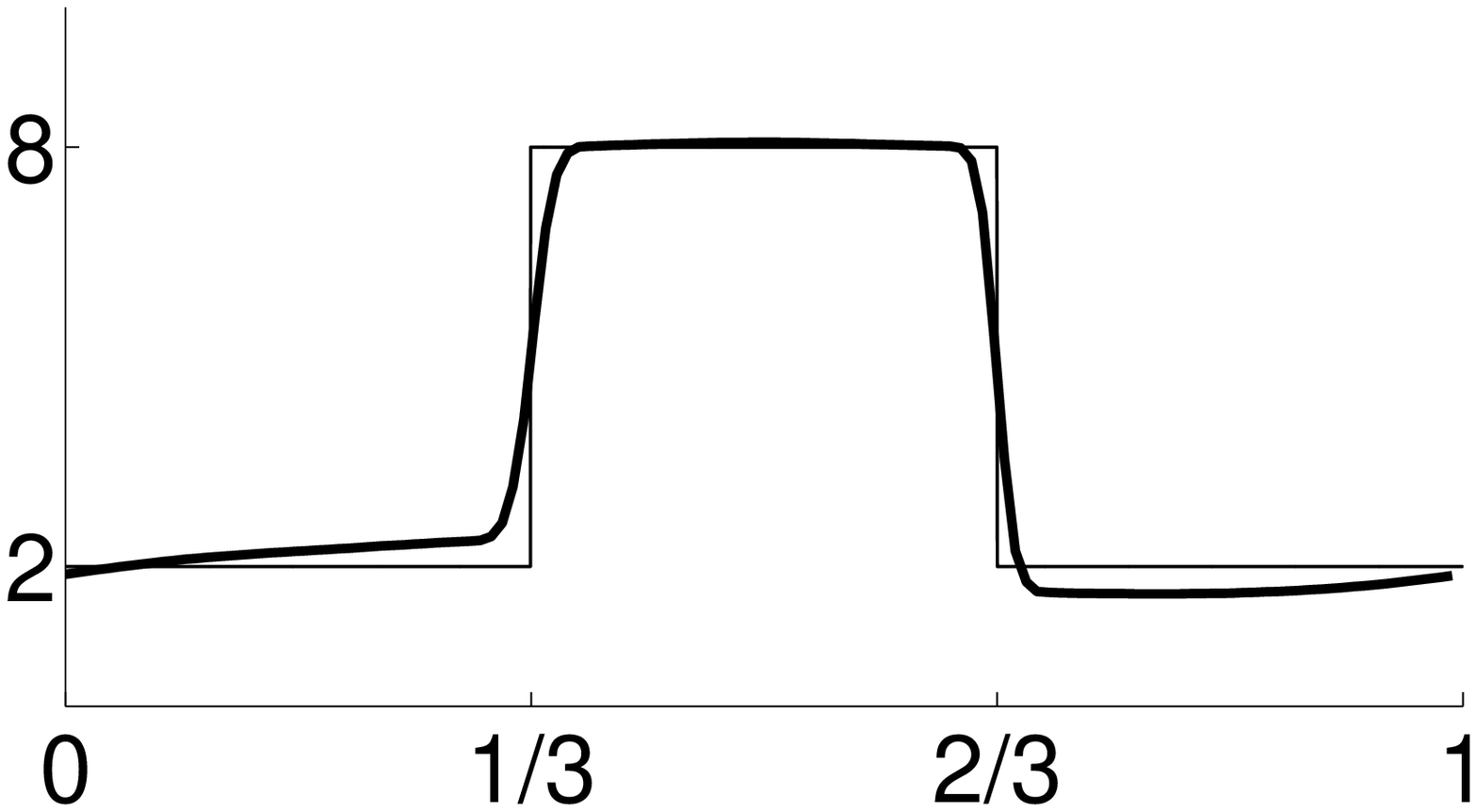}}
\put(165,165){$N=128$}
\put(290,90){\epsffile{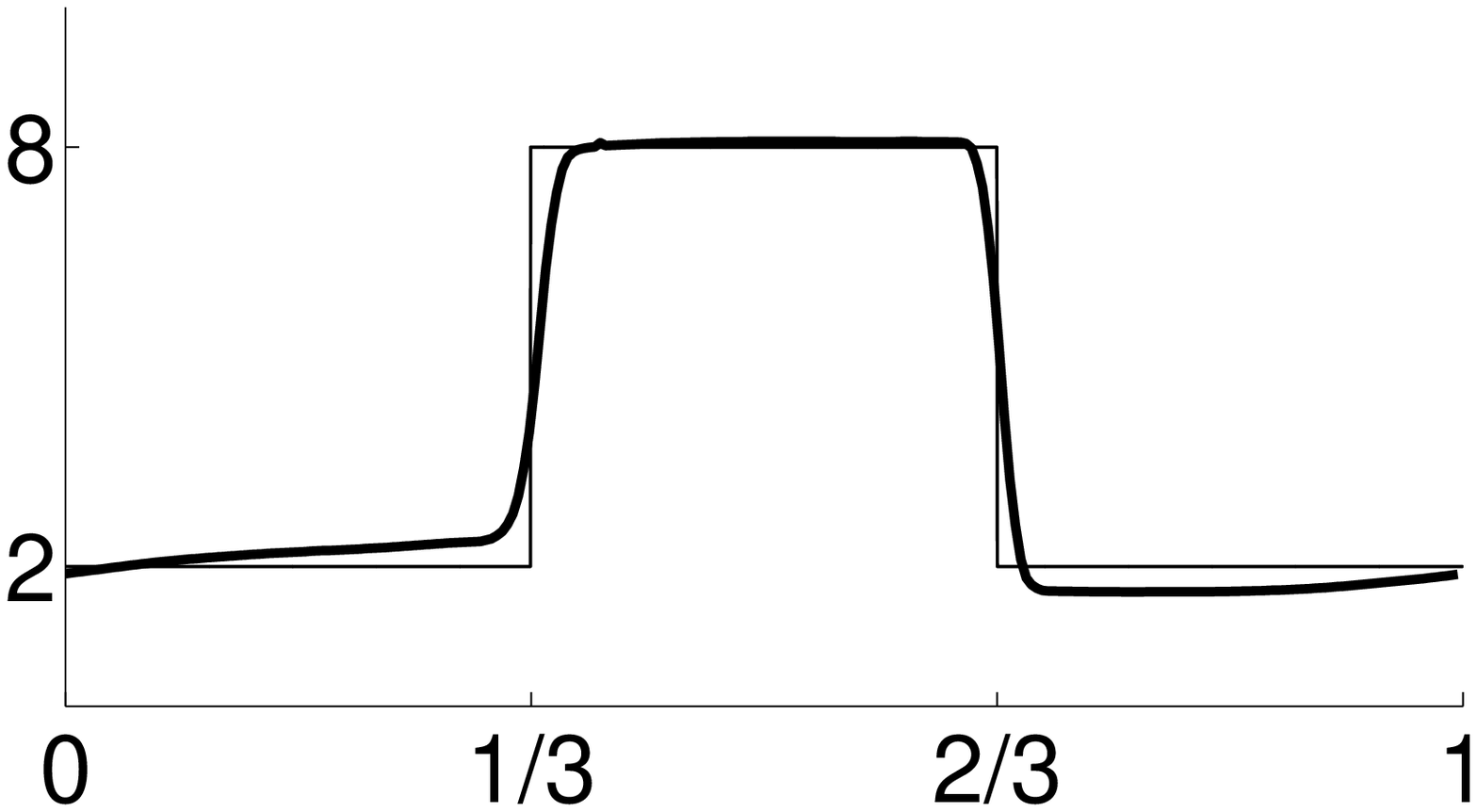}} 
\put(290,165){$N=256$}

\put(40,10){\epsffile{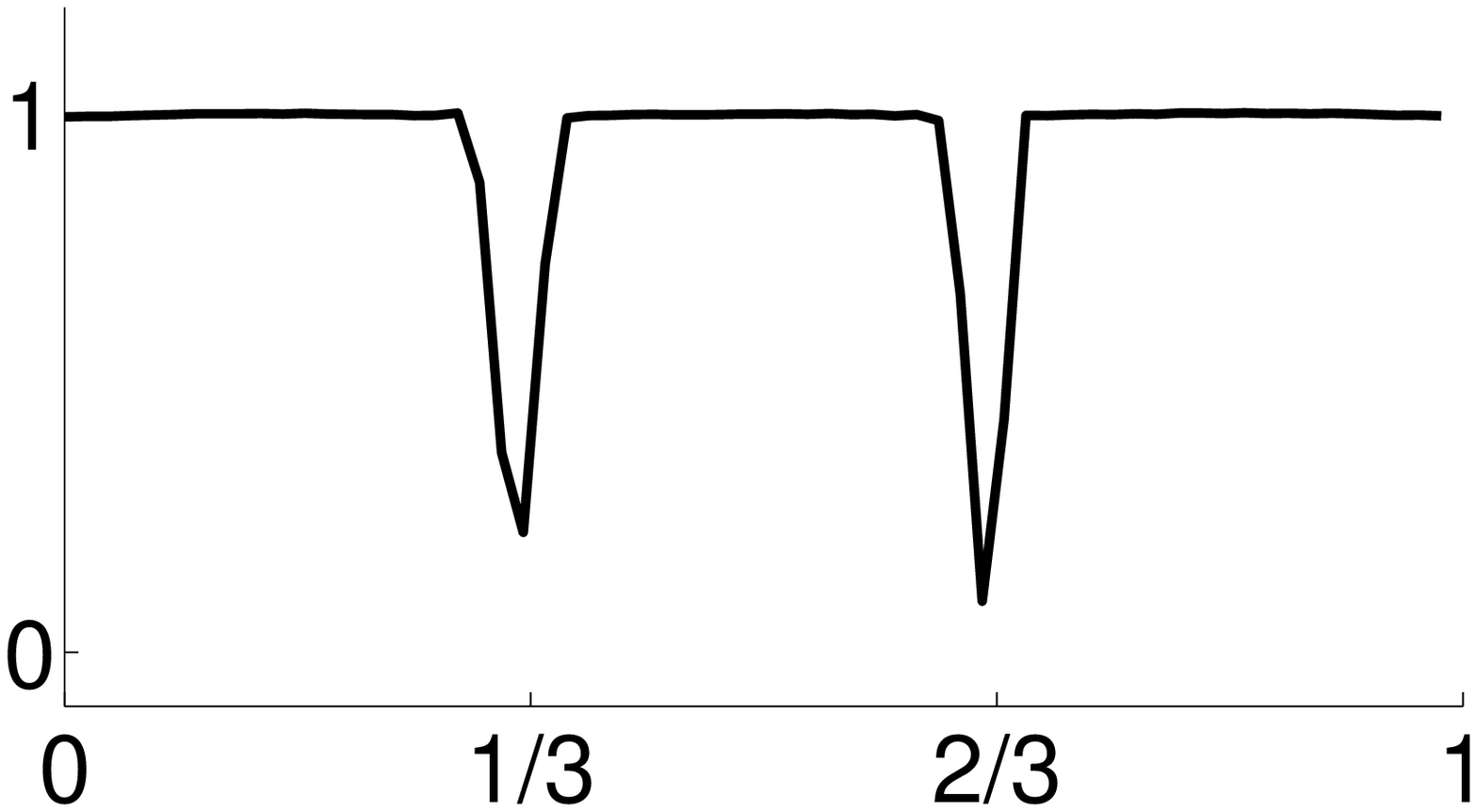}} 
\put(165,10){\epsffile{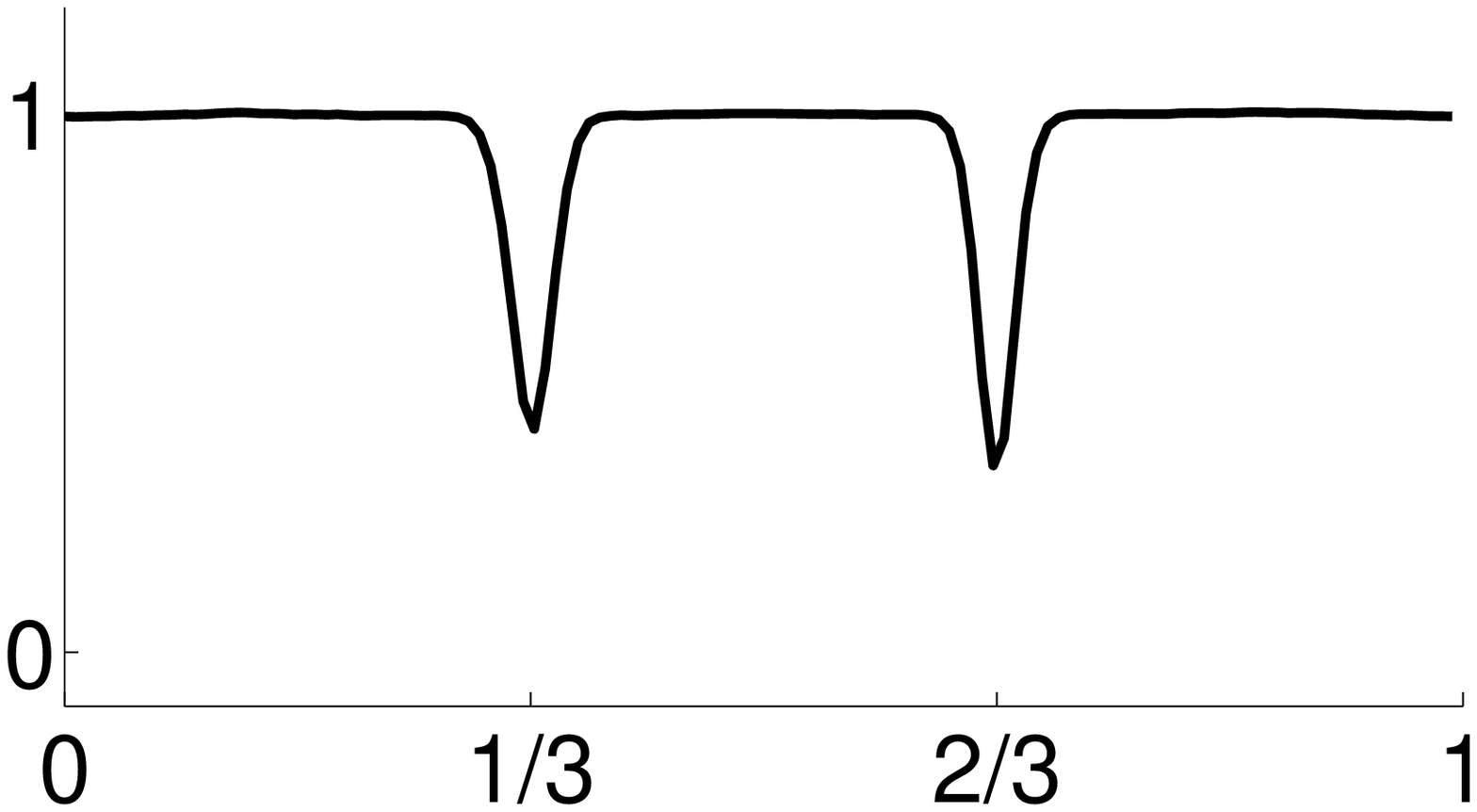}}
\put(290,10){\epsffile{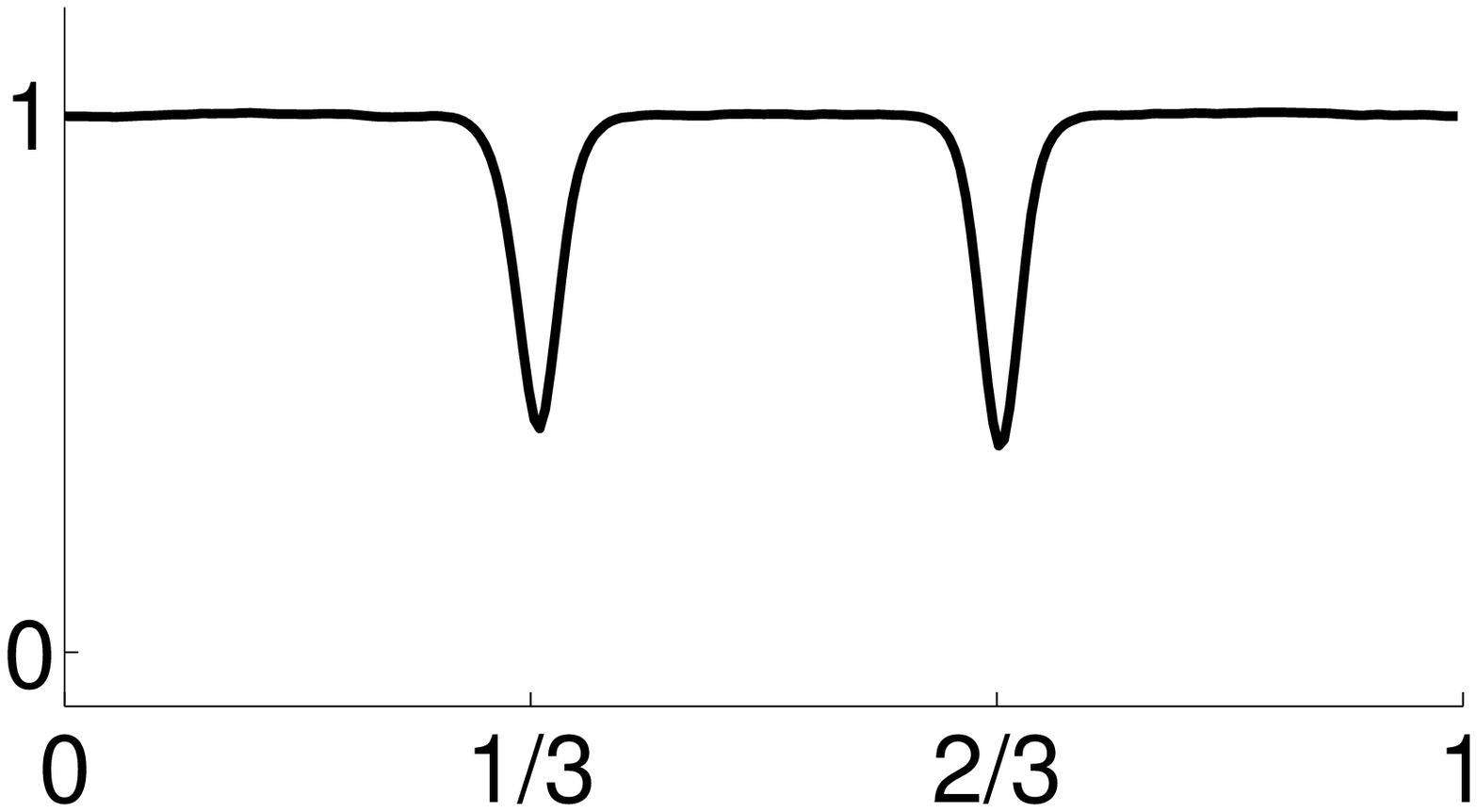}} 

\end{picture}
\caption{All the plots in this figure are obtained with the choice $\epsilon = 3\cdot 10^{-4}$ and $k=7$. 
Top row: the CM estimates $u^{CM}_{kn}$ with $n=6,7,8$ (thick line) and the true signal (thin line)
Bottom row: the CM estimates $v^{CM}_{kn}$.}
\label{fig3}
\end{figure}

\subsection{Discussion}

We have computed the CM estimates in relatively low dimensions (highest being $N=256$). This is due to the
long computational times of MCMC algorithms. The computational times can be improved with more sophisticated algorithm design,
e.g., parallelization. Furthermore, we expect MCMC methods to become much feasible in the future due to evolution
of computers.

It is evident from Figures \ref{fig2} and \ref{fig3} that 
the sharpness of edges in the CM estimates can be controlled via $\epsilon$ and
the CM estimates $u^{CM}_{kn}$ seem stable with respect to $n$.
The results concerning $u^{CM}_{kn}$ fit well to our
expectations of the true CM estimate being a slightly smoothened approximation of the real signal
represented in Figure \ref{fig_exact}.
Considering the relatively large noise in the measurement we conclude that
the method estimates the true signal $u$ well.
However, one can notice changes in functions $v^{CM}_{kn}$.
First of all, given larger value of $N$ the functions $v^{CM}_{kn}$ become smoother.
This phenomena is less visible with smaller values of $\ep$ but note that we have not proved
what the limiting estimates are exactly.
The author expects this phenomena to stabilize with higher values of $N$ but it should be checked in the future studies.
Second, given smaller value of $\ep$ the maximum of $|v^{CM}_{kn}-1|$ becomes smaller.
Although the asymptotic analysis of taking $\ep$ to zero was not considered in this paper
we expect that some coupling of $N$ and $\ep$ need to be made for algorithm to work properly asymptotically
with respect to $\ep$.
In the deterministic minimization problems of 
discrete Ambrosio--Tortorelli functionals one typically needs to 
assume that $N(\ep) \ep^2 \to \infty$ when $\ep$ goes to zero (see e.g. \cite{BeCo})

We conclude this discussion by pointing out
that we have not used any ad-hoc weighting of the prior or likelihood information. 
This additional flexibility of the algorithm can be achieved by scaling the covariances of $U$ or $V$ with a constant. \\

{\bf Acknowledgements:} This work was supported by Emil Aaltonen 
foundation, Graduate School of Inverse problems (Academy of Finland)
and Finnish Centre of Excellence in Inverse Problems Research (Academy of Finland).
The author would like to thank Petteri Piiroinen, Hanna Pikkarainen and Samuli Siltanen for
various useful discussions and is grateful to the anonymous referees for careful examination
of the manuscript.

\addcontentsline{toc}{section}{Bibliography}

\bibliographystyle{amsalpha}

\medskip

{\it E-mail address: }Tapio.Helin@tkk.fi

\end{document}